\documentclass{lmcs}
\pdfoutput=1
\usepackage[utf8]{inputenc}

\usepackage{lastpage}
\lmcsdoi{21}{1}{15}
\lmcsheading{}{\pageref{LastPage}}{}{}%
{May~25,~2023}{Feb.~11,~2025}{}

\usepackage[all,2cell]{xy}

\usepackage{graphicx}
\usepackage{amsmath}
\usepackage{amsthm}
\usepackage{amsfonts}
\usepackage{amssymb}
\usepackage{graphicx}
\usepackage{quiver}
\usepackage[colorlinks=true, allcolors=blue]{hyperref}
\usepackage{todonotes}
\theoremstyle{plain} 
\newtheorem{theorem}[thm]{Theorem}
\newtheorem{corollary}[thm]{Corollary}

\newtheorem{proposition}[thm]{Proposition}

\theoremstyle{definition} 
\newtheorem{definition}[thm]{Definition}
\newtheorem{remark}[thm]{Remark}
\newtheorem{example}[thm]{Example}

\makeatletter
\def\ifenv#1{
	\def\@tempa{#1}%
	\ifx\@tempa\@currenvir
	   \expandafter\@firstoftwo
	 \else
	   \expandafter\@secondoftwo
	\fi
}

\newcommand{\thlabel}[1]{%
	\label{#1}
	\begingroup
	\ifenv{theorem}{
		\def\@currentlabel{Theorem \ref{#1}}
		\label{#1-th}
	}
	{\ifenv{unprovedtheorem}{
		\def\@currentlabel{Theorem \ref{#1}}
		\label{#1-th}
	}
 	{\ifenv{proposition}{
		\def\@currentlabel{Proposition \ref{#1}}
		\label{#1-th}
	}
	{\ifenv{unprovedproposition}{
		\def\@currentlabel{Proposition \ref{#1}}
		\label{#1-th}
	}
	{\ifenv{lemma}{
		\def\@currentlabel{Lemma \ref{#1}}
		\label{#1-th}
	}
	{\ifenv{corollary}{
		\def\@currentlabel{Corollary \ref{#1}}
		\label{#1-th}
	}
	{\ifenv{remark}{
		\def\@currentlabel{Remark \ref{#1}}
		\label{#1-th}
	}
	{\ifenv{example}{
		\def\@currentlabel{Example \ref{#1}}
		\label{#1-th}
	}
	{\ifenv{definition}{
		\def\@currentlabel{Definition \ref{#1}}
		\label{#1-th}
	}
	{\ifenv{thmC}{
		\def\@currentlabel{Theorem \ref{#1}}
		\label{#1-th}
	}
	{\GenericWarning{No Environment Found}}}}}}}}}}}
	\endgroup
}

\newcommand{\thref}[1]{\ref{#1-th}}

\makeatother

\newcommand{\ovln}[1]{\overline{#1}}

\newcommand{\angbr}[2]{\langle #1,#2 \rangle}

\newcommand{\freccia}[3]{\xymatrix@1{#2 \colon #1  \ar[r] &  #3}}
\newcommand{\arrow}[3]{#2 \colon #1  \to #3}

\newcommand{\arrowup}[3]{\xymatrix@1{ #1  \ar[r]^{#2} &  #3}}

\newcommand{\pbmorph}[2]{#1^{\ast}#2}

\newcommand{\twomorphism}[6]{\xymatrix{   
#1^{\op} \ar[rrd]^#2_{}="a" \ar[dd]_{#3^{\op}}\\
&& \pos\\
#5^{\op}  \ar[rru]_#6^{}="b"
\ar_{#4}  "a";"b"}}
\newcommand{\quadratocomm}[8]{ \xymatrix@+1pc{  
#1 \ar[r]^{#5} \ar[d]_{#6} & #2 \ar[d]^{#7} \\
#3 \ar[r]_{#8} & #4 
}}

\newcommand{\pullbackcorner}[1][ul]{\save*!/#1+1.2pc/#1:(1,-1)@^{|-}\restore}

\def\pr{\pi}
\def\id{\operatorname{ id}}
\def\op{\operatorname{ op}}

\def\dom{\operatorname{ dom}}

\def\mC{\mathcal{C}}

\def\mD{\mathcal{D}}

\def\pos{\mathsf{Pos}}
\newcommand{\pca}[1]{\mathbb{#1}}
\newcommand{\subpca}[1]{\mathbb{#1}'}

\def\set{\mathsf{Set}}

\newcommand{\doctrine}[2]{\xymatrix@1{#2 \colon #1^{\op}  \ar[r] & \pos }}

\newcommand{\dial}[1]{\mathfrak{Dial}(#1)}
\newcommand{\dialfull}[1]{\mathfrak{Dial}_{\mathsf{f}}(#1)}

\newcommand{\compex}[1]{{#1}^{\exists}}
\newcommand{\compun}[1]{{#1}^{\forall}}
\newcommand{\compexfull}[1]{{#1}^{\exists_{\mathsf{f}}}}
\newcommand{\compunfull}[1]{{#1}^{\forall_{\mathsf{f}}}}

\newcommand{\leexcomp}{\le_{\exists}}
\newcommand{\lefullunicomp}{\le_{\forall_{\mathsf{f}}}}
\newcommand{\lefullexcomp}{\le_{\exists_{\mathsf{f}}}}

\newcommand{\function}[2]{\colon #1 \to #2}
\newcommand{\pfunction}[2]{:\subseteq #1 \to #2}
\newcommand{\pmfunction}[2]{:\subseteq #1 \rightrightarrows #2}

\newcommand{\ran}{\operatorname{ran}}

\newcommand{\powerset}{\raisebox{0.6mm}{\Large\ensuremath{\wp}}}

\newcommand{\st}{:}

\newcommand{\pcadoctrine}{\pca{A}^{(-)}}
\newcommand{\Mpcadoctrine}{\mathfrak{T}}
\newcommand{\Mpcaorder}{\le_{\mathsf{T}}}

\newcommand{\realizabilityOrder}{\le}
\newcommand{\MedvedevOrder}{\le_\mathrm{M}}
\newcommand{\MedvedevDoctrine}{\mathfrak{M}}

\newcommand{\medvedevreducible}{\le_\mathrm{M}}

\newcommand{\Mucpcadoctrine}{\pca{A}_w}
\newcommand{\Mucpcaorder}{\le_{\mathsf{dw}}}
\newcommand{\MuchnikOrder}{\le_w}
\newcommand{\MuchnikDoctrine}{\mathfrak{M}_w}

\newcommand{\WeiElementary}{\mathfrak{eW}}

\newcommand{\WeiOrderElementary}{\le_\mathsf{dW}}

\newcommand{\WeiDoctrine}{\mathfrak{W}}
\newcommand{\WeiDoctrineOrder}{\le_\mathsf{W}}

\newcommand{\barAsm}[1]{|#1|}
\newcommand{\relAsm}[1]{\Vdash_{#1}}

\newcommand{\strWeiElementary}{\mathfrak{sW}}

\newcommand{\newstrWeiDoctrine}{\mathfrak{SW}}
\newcommand{\strWeiOrder}{\le_\mathsf{dsW}}
\newcommand{\newstrWeiOrder}{\le_\mathsf{sW}}

\newcommand{\rweireducibile}{\le_{\mathrm{rW}}}
\newcommand{\rweiequiv}{\equiv_{\mathrm{rW}}}
\newcommand{\RWeiElementary}{\mathfrak{erW}}
\newcommand{\RWeiOrderElementary}{\le_\mathsf{drW}}
\newcommand{\RWeiDoctrine}{\mathfrak{rW}}
\newcommand{\RWeiDoctrineOrder}{\le_\mathsf{rW}}

\newcommand{\extweireducible}{\le_{\mathsf{extW}}}
\newcommand{\extstrongweireducible}{\le_{\mathsf{extsW}}}
\newcommand{\EWeiElementary}{\mathfrak{etW}}
\newcommand{\EWeiOrderElementary}{\le_\mathsf{dextW}}
\newcommand{\EWeiDoctrine}{\mathfrak{tW}}
\newcommand{\EWeiDoctrineOrder}{\le_\mathsf{extW}}

\newcommand{\Assemblies}{\mathsf{Asm}}
\newcommand{\Modest}{\mathsf{Mod}}

\newcommand{\weireducible}{\le_{\mathrm{W}}}
\newcommand{\strongweireducible}{\le_{\mathrm{sW}}}

\newcommand{\str}[1]{( #1 )}
\newcommand{\pairing}[1]{\langle #1 \rangle}
\newcommand{\coding}[1]{\langle #1 \rangle}
\newcommand{\concat}{\smash{\raisebox{.9ex}{\ensuremath\smallfrown} }}

\newcommand{\repmap}[1]{\delta_{#1}}

\newcommand{\Baire}{\mathbb{N}^\mathbb{N}}

\newcommand{\manlio}[1]{\todo[color=cyan, inline]{Manlio: #1}}

\newcommand{\subtraction}{\backslash}

\newcommand{\PcaCat}{\mathsf{Mod}_{\mathsf{p}}}
\newcommand{\extPcaCat}{\mathsf{extAsm}}
\newcommand{\partAsm}{\mathsf{ParAsm}}

\newcommand{\Hom}{\mathsf{Hom}}
\newcommand{\fprAsm}{\mathsf{fst}}
\newcommand{\sprAsm}{\mathsf{snd}}
\newcommand{\pair}{\mathsf{pair}}
\newcommand{\name}{\Vdash}
\newcommand{\D}{\mathfrak{D}}
\newcommand{\support}[1]{\| #1\|}

\usepackage{tikz-cd}

\usepackage{url}
\usepackage{doi}
\tikzcdset{%
    triple line/.code={\tikzset{%
        double distance = 3pt, 
        double=\pgfkeysvalueof{/tikz/commutative diagrams/background color}}},
    Rrightarrow/.code={\tikzcdset{triple line}\pgfsetarrows{tikzcd implies cap-tikzcd implies}}
}

\begin{document}

\title[Categorifying Computable Reducibilities]{Categorifying Computable Reducibilities}
\author[D.~Trotta]{Davide Trotta\lmcsorcid{https://orcid.org/0000-0003-4509-594X}}[a]	
\address{University of Padova, Italy}	
\email{trottadavide92@gmail.com}  

\author[M.~Valenti]{Manlio Valenti\lmcsorcid{0000-0003-0351-3058}}[b]	
\address{University of Wisconsin - Madison, USA}
\curraddr{Swansea University, UK}
\email{manliovalenti@gmail.com}
\thanks{
Davide Trotta's research has been partially supported by the Italian MIUR  
project PRIN 2017FTXR7S \emph{IT-MATTERS (Methods and Tools for Trustworthy Smart Systems)}. Manlio Valenti's research was partially supported by the Italian PRIN 2017 Grant \emph{Mathematical Logic: models, sets, computability}. 
Valeria de Paiva and Davide Trotta are grateful to the Hausdorff Research Institute for Mathematics in Bonn, Germany, for hosting us as part of the trimester ``Prospects of Formal Mathematics,'' funded by the Deutsche Forschungsgemeinschaft (DFG, German Research Foundation) under Germany's Excellence Strategy – EXC-2047/1 – 390685813.
The authors would like to thank Jonas Frey, Takayuki Kihara, and Arno Pauly for useful conversations on the topics of the paper. They also thank the anonymous referees for their careful reading of the paper and many valuable suggestions.}
\author[V.~de Paiva]{Valeria de Paiva\lmcsorcid{https://orcid.org/0000-0002-1078-6970}}[c]
\address{Topos Institute, California, USA}
\email{valeria@topos.institute}

\begin{abstract}
This paper presents categorical formulations of Turing, Medvedev, Muchnik, and Weihrauch reducibilities in Computability Theory, utilizing Lawvere doctrines. While the first notions lend themselves to a smooth categorical presentation, essentially dualizing the traditional idea of realizability doctrines, Weihrauch reducibility and its extensions to represented and multi-represented spaces require a separate investigation.

Our abstract analysis of these concepts highlights a shared characteristic among all these reducibilities. Specifically, we demonstrate that all these doctrines stemming from computability concepts can be proven to be instances of completions of quantifiers for doctrines, analogous to what occurs for doctrines for realizability.
As a corollary of these results, we will be able to formally compare Weihrauch reducibility with the dialectica doctrine constructed from a doctrine representing Turing degrees.
 
\end{abstract}

\maketitle

\section{Introduction}
Categorical methods and language have been employed in many areas of Mathematics. In Mathematical Logic, they are widely used in Model Theory and Proof Theory. In Recursion or Computability Theory,  there is a long tradition of categorical methods in Realizability studies, expounded in van Oosten's book  \cite{van_Oosten_realizability}.
One of the key tools employed in such a setting is the notion of \emph{tripos}, introduced by Hyland, Johnstone, and Pitts in \cite{hyland89}, which is a specific instance of the notion of \emph{hyperdoctrine} introduced by Lawvere ~\cite{lawvere1969,lawvere1969b,lawvere1970} to synthesize the structural properties of logical systems.

The use of categorical methods in Realizability includes, for example, Hofstra's work  \cite{Hofstra2006}, where he proved that most well-known realizability-like triposes, (e.g. the ``effective" tripos~\cite{hyland89}, the ``modified realizability" tripos~\cite{VANOOSTEN1997273} and the ``dialectica'' tripos~\cite{Biering_dialecticainterpretations}) are instances of a more general notion of tripos associated to a given ordered partial combinatory algebra (PCA) equipped with a filter. It is worth recalling that the notion of PCA presents a generalization of both Kleene's first and second models.
His point was that all these triposes can be presented as ``triposes for a given PCA", hence all these notions differ only in the choice of the associated PCA, and we could say that all realizability is relative to a choice of a PCA equipped with a filter. 
%

Over the years, several authors observed that realizability triposes are instances of a free construction adding ``generalized existential quantifiers'' to a given doctrine. For instance, we refer to Hofstra's work \cite{Hofstra2006}, Trotta and Maietti's work \cite{trottamaietti2020}, and Frey's works \cite{Frey2014AFS,Frey2020}. So every realizability tripos is obtained by freely adding left adjoints along the class of all morphisms of the base category. Hence, combining this result with the previous analysis by Hofstra, we have that realizability-like triposes are instances of the generalized existential completion. These results show the abstract, structural property that lies behind all the various forms of realizability.

Despite the long tradition of studying realizability and its variants from a categorical perspective, a systematic and in-depth analysis of computability-like reducibilities, such as Turing \cite{Odifreddi92, Soare1987}, Medvedev~\cite{Sorbi1996,Hinman2012}, Muchnik~\ \cite{Hinman2012}, and Weihrauch reducibilities~\cite{Brattka_2011}, from a categorical perspective is still lacking. 
A categorical presentation of these notions is useful to highlight some of their abstract structural properties. Moreover, this work provides a common language for the categorical logic community and the computability theorists. These are two fields that traditionally employ very different languages and notations, so a categorical understanding of reducibility in computability with categorical logic descriptions in terms of doctrines has the positive effect of fostering collaborations between the two communities.

A first approach to Medvedev and Muchnik reducibility via hyperdoctrines has been introduced in \cite{Kuyper2015}, while Weihrauch reducibility for assemblies (or multi-represented spaces) has been introduced only very recently in \cite{Kihara2022rethinking, SchroederCCA2022} through the notion of realizer-based Weihrauch reducibility. In \cite{Bauer2021}, Bauer introduced an abstract notion of reducibility between predicates, called instance reducibility, which commonly appears in reverse constructive mathematics. In a relative realizability topos, the instance degrees correspond to a generalization of (realizer-based) Weihrauch reducibility, called extended Weihrauch degrees, and the ``classical'' Weihrauch degrees correspond precisely to the $\lnot\lnot$-dense modest instance degrees in Kleene-Vesley realizability. Upon closer inspection, it is not hard to check that realizer-based Weihrauch reducibility is a particular case of Bauer's notion.

In this paper, we want to extend the use of the categorical tools of doctrines to computability, by presenting a specific doctrine for every previously mentioned notion of reducibility and studying their universal properties. In detail, we first introduce a doctrine for each of the notions of Turing, Medvedev, and Muchnik reducibility. After showing that these doctrines provide a proper categorification of these notions, we prove that both Medvedev and Muchnik reducibility are instances of the categorical construction called ``universal completion'' which freely adds ``generalized universal quantifiers'' to a given doctrine. This construction is a natural generalization of the ``pure universal completion'' used to characterize dialectica doctrines in \cite{trotta-lfcs2022,trotta23TCS}. As a corollary of these results, we show that the Medvedev doctrine is obtained as the full universal completion of the doctrine of Turing degrees.

We then focus on the notion of Weihrauch reducibility (and its strong version) from a categorical perspective. Following along the same lines, we start our analysis by defining a doctrine abstracting the ordinary notion of Weihrauch reducibility, and then we prove that this doctrine can be obtained as the pure existential completion of a more basic doctrine. The crucial difference between this setting and the previous one concerns the base category of the doctrine we consider. Indeed, while for Medvedev and Muchnik reducibility, the base category of the doctrines is the category $\set$ of sets and functions, for Weihrauch we will use the category $\PcaCat (\pca{A},\pca{A}')$ whose objects are subsets of the PCA $\pca{A}$, and whose morphisms are computable functions. This change of perspective is necessary to fully abstract the details of a Weihrauch reduction (in particular, it captures the role of the \emph{forward functional}). 

The last part of our work is devoted to studying Weihrauch reducibility in the more general context of partial multi-valued functions on represented spaces, a common setting for people working in computable analysis. However, the problem of generalizing our previous approach to this setting is non-trivial. 

In order to fully (and smoothly) abstract Weihrauch reducibility for represented spaces in the language of doctrines, we introduce a new base-category $\extPcaCat (\pca{A},\pca{A}')$ for our doctrines, that will turn out to be equivalent to the ordinary category of partition assemblies. The objects of $\extPcaCat (\pca{A},\pca{A}')$ are assemblies, but the morphisms are better suited to abstract the properties of the forward functional. Even if the category of partition assemblies is well-known and studied in the literature, using the category $\extPcaCat (\pca{A},\pca{A}')$ as base-category for doctrines abstracting Weihrauch reducibility has two main advantages: from a conceptual point of view, its objects and morphisms have an immediate and clear connection with the usual notions involved in the ordinary presentation of generalizations of Weihrauch reducibility to assemblies. Then, from a purely technical perspective, employing $\extPcaCat (\pca{A},\pca{A}')$ will make the proofs of our main results quite smooth.

Once the issue concerning the base category is resolved and the doctrines are appropriately defined, we finally demonstrate that all these doctrines are instances of pure existential completion, following the same lines adopted for the first Weihrauch doctrine we introduced.

Our abstract analysis systematizes all these variants of Weihrauch reducibility and shows a clear connection between them. In particular, we highlight how extended Weihrauch reducibility can be seen as the most general variant, and all the others can be obtained from it by considering suitable restrictions of this doctrine.

We conclude our work by highlighting an interesting connection between Medvedev and extended Weihrauch reducibility, and discussing how the notion of extended strong Weihrauch is related to G\"odel's Dialectica interpretation \cite{goedel1986}. In particular, our presentation allows us to provide a formal explanation of the resemblance between the structure of Dialectica categories and some known notions of computability (extended Weihrauch reducibility), which has been observed by several authors over the years.

The outline of this paper is as follows: in Section~\ref{sec:basic notions} and Section~\ref{sec:docs}, we recall the background notions of partial combinatory algebra and quantifier completions. In Section~\ref{sec:medvedev} we reformulate Medvedev and Muchnik reducibility in categorical terms. In Section~\ref{sec:weihrauch_doctrine} we categorify Weihrauch reducibility and, in Section~\ref{sec:gen_weihrauch}, we consider the just mentioned generalizations of Weihrauch reducibility. Finally, we draw some conclusions in Section~\ref{sec:conclusions}.

\section{Partial combinatory algebras
}\thlabel{sec:basic notions}

Realizability theory originated with Kleene’s interpretation of intuitionistic number theory~\cite{kleene1945} and has since developed into a large body of work in logic and theoretical computer science. We focus on two basic flavors of realizability, number realizability, and function realizability, which were both due to Kleene. 

In this section, we recall some standard notions within realizability and computability. We follow the approach suggested by van Oosten~\cite{van_Oosten_realizability}, as we want to fix a suitable notation for both category theorists and computability logicians. 

We describe \emph{partial combinatory algebras} and discuss some important examples.  For more details, we refer the reader to van Oosten's work on categorical realizability (see \cite{van_Oosten_realizability} and the references therein). 

We start by introducing the basic concept of a \emph{partial applicative structure} or PAS, due to Feferman, which may be viewed as a universe for computation. 

\begin{definition}[PAS]
    A  \textbf{partial applicative structure}, or PAS for short, is a set $\pca{A}$ equipped with a 
    \textbf{partial} binary operation 
$\cdot\pfunction{\pca{A}\times \pca{A}}{\pca{A}}$.
\end{definition}

Some conventions and terminology: given two elements $a,b$ in $\pca{A}$, we think of $a\cdot b$ (which we often abbreviate $ab$) as  ``$a$ applied to $b$". The partiality of the operation $\cdot$ means that this application need not always be defined. We write $f\pfunction{A}{B}$ to say that $f$ is a partial function with domain a subset of $A$, $\dom(f)\subseteq A$ and range $\ran(f)=B$.
If $(a,b)\in \dom(\cdot)$, that is, when the application is defined, then we write $a\cdot b \downarrow$ or $ab\downarrow$. 

We usually omit brackets, assuming associativity of application to the left. Thus $abc$ stands for $(ab)c$. Moreover, for two expressions $x$ and $y$ we write $x \simeq  y$ to indicate that $x$ is defined whenever $y$ is, in which case they are equal.

Even though these partial applicative structures do not possess many interesting properties (they have no axioms for application), they already highlight one of the key features of combinatorial structures, namely the fact that we have a domain of elements that can act both as functions and as arguments, just as in untyped $\lambda$-calculus. 
This behavior can be traced back to Von Neumann's idea that programs (functions, operations) live in the same realm and are represented in the same way as the data (arguments) that they act upon.
In particular, programs can act on other programs.

\begin{definition}[PCA]\thlabel{def: PCA}
    A \textbf{partial combinatory algebra} (PCA) is a PAS $\pca{A}$  for which there exist elements $k,s\in \pca{A}$ such that for all $a,b,c\in \pca{A}$ we have that
\[k a \downarrow \mbox{ and } kab\simeq a\] 
and
\[s a \downarrow \mbox{, } sab\downarrow \mbox{, and } sabc\simeq ac (bc)\] 
\end{definition}

The elements $k$ and $s$ are generalizations of the homonymous combinators in Combinatory Logic. Note that appropriate elements $k, s$ are not considered part of the structure of the partial combinatory algebra, so they need not be preserved under homomorphisms.

Every PCA $\pca{A}$ is combinatory complete in the sense of \cite{Feferman75,van_Oosten_realizability}, namely: for every term $t(x_1,\dots,x_{n+1})$ built from variables $x_1,\dots,x_{n+1}$, constants $\bar{c}$ for $c\in \pca{A}$, and application operator $\cdot$, there is an element $a\in \pca{A}$ such that for all elements $b_1,\dots,b_{n+1}\in \pca{A}$ we have that $ab_1\cdots b_n \downarrow$ and $ab_1\cdots b_{n+1}\simeq t(b_1,\dots,b_{n+1})$. 

In particular, we can use this result and the elements $k$ and $s$ to construct elements $\pair,  \fprAsm, \sprAsm $ of $\pca{A}$ so that $(a,b)\mapsto \pair \cdot a\cdot b$ is an injection of $\pca{A}\times \pca{A}$ to $\pca{A} $ with left inverse $a\mapsto (\fprAsm\cdot a,\sprAsm \cdot a)$. Hence, we can use $\pair ab$ as an element of $\pca{A}$ which codes the pair $(a, b)$. For this reason, the elements $\pair, \fprAsm ,\sprAsm $ are usually called \emph{pairing} and \emph{projection} operators. For the sake of readability, we write $\pairing{a,b}$ in place of $\pair ab$ (as it is more customary in computability theory).

Using $s$ and $k$, we can prove the analogues of the Universal Turing Machine (UTM) and the SMN theorems in computability in an arbitrary PCA. 
    A PCA $\pca{A}$ is called \textbf{extensional} if  for all $ x$ in $\pca{A}$,  $(ax\simeq bx)$ implies $a=b$. 

By definition, in every extensional PCA, if two elements represent the same partial function, then they must be equal.

Next we recall the notion of \emph{elementary sub-PCA} (see for example \cite[Sec.\ 2.6.9]{van_Oosten_realizability}). Many definitions in the computability context refer to a concept and a subset of the given concept, as one needs to pay attention to the computable functions (and elements) included in the original concept.
\begin{definition}[elementary sub-PCA] \thlabel{def:sub-pca}
Let $\pca{A}$ be a PCA. A subset $\subpca{A}\subseteq \pca{A}$ is called an \textbf{elementary sub-PCA} of $\pca{A}$ if $\subpca{A}$ is a PCA with the partial applicative structure induced by $\pca{A}$ (namely, the elements $k$ and $s$ as in \thref{def: PCA} can be found in $\pca{A}'$). In particular, it is closed under the application of $\pca{A}$, that means: if $a,b\in \subpca{A}$ and $ab\downarrow$ in $\pca{A}$ then $ab\in \subpca{A}$.
\end{definition}

In particular, elements of the sub-PCA will play the role of the computable functions.

\begin{example}[\emph{Kleene’s first model}]\label{ex: K_1}
Fix an effective enumeration $(\varphi_a)_{a\in\mathbb{N}}$ of the partial recursive functions $\mathbb{N}\to\mathbb{N}$ (i.e.\ a G\"odel numbering). The set $\pca{N}$ with partial recursive application $(a,b)\mapsto \varphi_a(b)$ is a PCA, and it is called \emph{Kleene’s first model} $\mathcal{K}_1$ (see e.g.\ \cite{Soare1987}).
\end{example}

\begin{example}
[\emph{Kleene’s second model}] The PCA $\mathcal{K}_2$ is often used for function realizability \cite[Sec.\ 1.4.3]{van_Oosten_realizability}. This PCA is given by the Baire space $\pca{N}^{\pca{N}}$, endowed with the product topology. 
The partial binary operation of application $\cdot\pfunction{\pca{N}^{\pca{N}}\times \pca{N}^{\pca{N}}}{\pca{N}^{\pca{N}}} $ corresponds to the one used in Type-$2$ Theory of Effectivity \cite{Weihrauch00}.
This can be described as follows. Let $\alpha[n]$ denote the string $\str{\alpha(0),\hdots, \alpha(n-1)}$. 
Every $\alpha\in \Baire$ induces a function $F_{\alpha}\pfunction{\Baire}{\mathbb{N}}$ defined as $F_{\alpha}(\beta)=k$ if there is $n\in\mathbb{N}$ such that $\alpha(\coding{\beta[n]})=k+1$ and $(\forall m<n)(\alpha(\coding{\beta[n]})=0)$, and undefined otherwise. The application $\alpha\cdot \beta$ can then be defined as the map $n\mapsto F_{\alpha}(\str{n}\concat \beta)$, where $\str{n}\concat \beta$ is the string $\sigma$ defined as $\sigma(0):=n$ and $\sigma(k+1):=\beta(k)$.

When working with Kleene's second model, we usually consider the elementary sub-PCA $\mathcal{K}_2^{rec}$ consisting of all the $\alpha\in\Baire$ such that $\beta\mapsto \alpha\cdot \beta$ is computable. For more details and examples, also of non-elementary sub-PCAs for Kleene's second model, we refer to \cite{Oosten2011PartialCA}.
\end{example}

\begin{remark}
Kleene's $\mathcal{K}_1$ and $\mathcal{K}_2$ presented above are not extensional, since there are many codes (programs) that compute the same function. In fact, every function has infinitely many representatives.
\end{remark}

\section{Categorical Doctrines}
\label{sec:docs}
We want to connect the notions of computability as described in the previous subsection to work on (logical) categorical doctrines in \cite{trotta-lfcs2022}. We recap only the essential definitions from the doctrines work in the text; further details and explanation can be found in \cite{trotta_et_al:LIPIcs.MFCS.2021.87, trotta-lfcs2022}.

Several generalizations of the notion of a (Lawvere) hyperdoctrine have been considered recently,  we refer, for example, to the works of Rosolini and Maietti \cite{maiettipasqualirosolini,maiettirosolini13a,maiettirosolini13b},  or to \cite{pitts02,hyland89} for higher-order versions. In this work, we consider a natural generalization of a hyperdoctrine, which we call simply a \emph{doctrine}.

\begin{definition}[doctrine]
A \textbf{doctrine} is a contravariant functor:
$$\doctrine{\mC}{P}$$ 
where the category $\mathcal{C}$ has finite products and $\pos$ is the category of (partially ordered sets or) posets.
\end{definition}
\begin{definition}[morphism of doctrines]
    A morphism of doctrines is a pair $\mathfrak{L}:=(F,\mathfrak{b})$
    $$\twomorphism{\mC}{P}{F}{\mathfrak{b}}{\mD}{R}$$
such that $\arrow{\mC}{F}{\mD}$ is a finite product preserving functor and: $$\arrow{P}{\mathfrak{b}}{R F^{\op}}$$ is a natural transformation.
\end{definition}

\begin{example}\thlabel{ex: doctrine PCA}
Let $\pca{A}$ be a PCA. 
We can define a functor $\doctrine{\set}{\pcadoctrine}$ assigning to a set $X$ the set  $\pca{A}^X$ of functions from $X$ to $\pca{A}$. 
By standard properties of PCAs, given two elements $\alpha,\beta\in \pca{A}^X$,  we have the following preorder: $\alpha\leq \beta$ if there exists an element $a\in \pca{A}$ such that for every $x\in X$ we have that $a\cdot \alpha(x)$ is defined and $a\cdot \alpha(x)=\beta(x)$.
The doctrine given by considering the poset reflection of such a preorder is the \emph{doctrine associated to the PCA $\pca{A}$}. Notice that this construction can be generalized in the context of relative realizability, see for example \cite[p.\  253]{Hofstra2006}. 
\end{example}
We recall the following example from \cite{pitts02,hyland89}.
\begin{example}\thlabel{ex: realizability tripos}
Given a PCA $\pca{A}$, we can consider the \textbf{realizability doctrine}\linebreak[4]$\doctrine{\set}{\mathcal{R}}$ 
over $\set$. For each set $X$, the partial order $(\mathcal{R}(X),\realizabilityOrder)$ is defined as the set of functions $\powerset(\pca{A})^X$ from $X$ to the powerset $\powerset(\pca{A})$ of $\pca{A}$. Given two elements $\alpha$ and $\beta$ of $\mathcal{R}(X)$, we say that $\alpha\realizabilityOrder \beta$ if there exists an element $\ovln{a}\in \pca{A}$ such that for all $x\in X$ and all $a\in \alpha (x)$, $\ovln{a}\cdot a$ is defined and it is an element of $\beta (x)$.  By standard properties of PCAs this relation is reflexive and transitive, i.e.\ it is a preorder.
Then $\mathcal{R} (X)$ is defined as the quotient of $\powerset(\pca{A})^X$ by the equivalence relation generated by the $\realizabilityOrder$.  The partial order on the equivalence classes $[\alpha]$ is the one induced by $\realizabilityOrder$. 
\end{example}

We also need to recall the general definitions of existential and universal doctrines.

\begin{definition}[$\mD$-existential/universal doctrines]\label{def full existential doctrine}
Let $\mC$ be a category with finite products, and let $\mD$ a class of morphisms of $\mC$ closed under composition, pullbacks, and identities. A  doctrine $\doctrine{\mC}{P}$ is $\mD$-\textbf{existential} (resp.\ $\mD$-\textbf{universal}) if, for every arrow $\arrow{X}{f}{A}$ of $\mD$, the functor
$$ \arrow{PA}{P_f}{PX}$$
has a left adjoint $\exists_{f}$ (resp.\ a right adjoint $\forall_{f}$), and these satisfy the Beck-Chevalley condition BC: 
for any pullback diagram
$$
\quadratocomm{X'}{A'}{X}{A}{f'}{h'}{h}{f}
$$
 and any $\beta$ in $P(X)$ the equality
$$\exists_{f'}P_{h'}\beta= P_h \exists_{f}\beta \,\,\,\, \textnormal{  (resp. }\forall_{f'}P_{h'}\beta= P_h \forall_{f}\beta\textnormal{  )}$$ holds. When $\mD$ is the class of all the morphisms of $\mC$ we say that the doctrine $P$ is \textbf{full existential} (resp.\ \textbf{full universal}), while when $\mD$ is the class of product projections, we will say that $P$ is \textbf{pure existential} (resp.\ \textbf{pure universal}).
\end{definition}

Next, we summarize the main properties of the generic full existential and universal completions in the following theorems and refer to \cite{trotta2020} for more details. 

\bigskip
\noindent
\textbf{Generalized existential completion}. Let $\doctrine{\mC}{P}$ be a doctrine and let $\mathcal{D}$ be a class of morphisms of $\mC$ closed under composition, pullbacks and containing identities.  For every object $A$ of $\mC$ consider the following preorder:
\begin{itemize}
\item \textbf{objects:} pairs $(\arrowup{B}{f\in \mathcal{D}}{A},\alpha)$, where $\arrow{B}{f}{A}$ is an arrow of $\mathcal{D}$ and $\alpha\in P(B)$.
\item \textbf{order:}  $(\arrowup{B}{f\in \mD}{A},\alpha)\leq (\arrowup{C}{g\in \mD}{A},\beta)$ if there exists an arrow $\freccia{B}{h}{C}$ of $\mC$ such that the diagram 
\[\xymatrix@+1pc{
& B\ar[d]^f\ar[ld]_h\\
C\ar[r]_{g} & A
}\]
commutes and
\[ \alpha\leq P_{h}(\beta).\]

\end{itemize}
It is easy to see that the previous construction gives a preorder. We denote by
$P^{\exists_{\mD}}(A)$ the partial order obtained by identifying two
objects when
$$(\arrowup{B}{h\in \mD}{A}, \alpha)\gtreqless (\arrowup{D}{f\in \mD}{A}, \gamma)$$
in the usual way. With a small abuse of notation, we identify an equivalence class with one of its representatives.

Given a morphism $\arrow{A}{f}{B}$ in $\mC$, let $P^{\exists_{\mD}}_f(\arrowup{C}{g}{B},\beta)$ be the object 
\[(\arrowup{D}{\pbmorph{f}{g}\in \mD}{A},\; P_{g^*f}(\beta) )\]
where $f^*g$ and $g^*f$ are defined by the pullback
\[\xymatrix@+1pc{
D\ar[d]_{f^*g} \pullbackcorner \ar[r]^{g^*f} &C\ar[d]^g\\
A \ar[r]_f & B.
}\]
The assignment $\doctrine{\mC}{P^{\exists_{\mD}}}$ is called the \textbf{generalized existential completion} of $P$. Following \cite{trotta2020,trottamaietti2020}, when $\mD$ is the class of all the morphisms of the base category, we will speak of \textbf{full existential completion}, and we will use the notation $\doctrine{\mC}{\compexfull{P}}$. Moreover, when $\mD$ is the class of all product projections, we will speak of \textbf{pure existential completion}, and we will use the notation $\doctrine{\mC}{\compex{P}}$.

\begin{thmC}[\cite{trotta2020}]
\label{theorem full existential comp}
The doctrine $P^{\exists_{\mD}}$ is $\mD$-existential. Moreover, for every  doctrine $P$ we have a canonical inclusion $\arrow{P}{\eta_P^{\exists_\mD}}{P^{\exists_{\mD}}}$ such that, for every morphism of  doctrines $\arrow{P}{\mathfrak{L}}{R}$, where $R$ is $\mD'$-existential and the functor between the bases sends arrows of $\mD$ into arrows of $\mD'$, there exists a unique (up to isomorphism) existential morphism doctrine (i.e.\ preserving existential quantifiers along $\mD$) such that the diagram
\[\xymatrix@+1pc{
& P^{\exists_{\mD}}\ar@{-->}[d]\\
P \ar[ru]^{\eta_P^{\exists_\mD}}\ar[r]_{\mathfrak{L}} & R
}\]
commutes.
\end{thmC}
\begin{example}
    Realizability doctrines are relevant examples of doctrines arising as full existential completions. The original observation of this result is due to Hofstra \cite{Hofstra2006}, while a more general analysis of doctrines arising as full existential completions can be found in \cite{trottamaietti2020,Frey2020}.
\end{example}
By dualizing the previous construction, we can define the $\mD$-\emph{universal completion} of a doctrine.

\bigskip
\noindent
\textbf{Generalized universal completion}. Let $\doctrine{\mC}{P}$ be a doctrine and let $\mathcal{D}$ be a class of morphisms of $\mC$ closed under composition, pullbacks and containing identities.  For every object $A$ of $\mC$ consider the following preorder:
\begin{itemize}
\item \textbf{objects:} pairs $(\arrowup{B}{f\in \mD}{A},\alpha)$, where $\arrow{B}{f}{A}$ is an arrow of $\mD$ and $\alpha\in P(B)$.
\item \textbf{order:}  $(\arrowup{B}{f\in \mD}{A},\alpha)\leq (\arrowup{C}{g\in \mD}{A},\beta)$ if there exists an arrow $\freccia{C}{h}{B}$ of $\mC$ such that the diagram 
\[\xymatrix@+1pc{
& B\ar[d]^f\\
C\ar[r]_{g}\ar[ru]^h & A
}\]
commutes and
\[ P_h(\alpha)\leq \beta .\]

\end{itemize}
Again, it is easy to see that the previous data gives us a preorder. We denote by
$P^{\forall_{\mD}}(A)$ the partial order obtained by identifying two
objects when
$$(\arrowup{B}{h\in \mD}{A}, \alpha)\gtreqless (\arrowup{D}{f\in \mD}{A}, \gamma)$$
in the usual way. As before, with a small abuse of notation, we identify an equivalence class with one of its representatives.

Given a morphism $\arrow{A}{f}{B}$ in $\mC$, let $P^{\forall_{\mD}}_f(\arrowup{C}{g\in \mD}{B},\beta)$ be the object 
\[(\arrowup{D}{\pbmorph{f}{g}\in \mD}{A},\; P_{g^*f}(\beta) )\]
where  $f^*g$ and $g^*f$ are defined by the pullback
\[\xymatrix@+1pc{
D\ar[d]_{f^*g} \pullbackcorner \ar[r]^{g^*f} &C\ar[d]^g\\
A \ar[r]_f & B.
}\]
The assignment $\doctrine{\mC}{\compunfull{P}}$ is called the $\mD$-\textbf{universal completion} of $P$. As before, when $\mD$ is the class of all the morphisms of the base category, we will speak of \textbf{full universal completion}, and we will use the notation $\doctrine{\mC}{\compunfull{P}}$. Moreover, when $\mD$ is the class of all product projections, we will speak of \textbf{pure universal completion}, and we will use the notation $\doctrine{\mC}{\compun{P}}$.

\begin{thmC}[\cite{trotta2020}]
\thlabel{theorem full universal comp}
The doctrine $P^{\forall_{\mD}}$ is $\mD$-universal. Moreover, for every doctrine $P$ we have a canonical inclusion $\arrow{P}{\eta_P^{\forall_\mD}}{P^{\forall_{\mD}}}$ such that, for every morphism of doctrines $\arrow{P}{\mathfrak{L}}{R}$, where $R$ is $\mD'$-universal and the functor between the bases sends arrows of $\mD$ into arrows of $\mD'$, there exists a unique (up to isomorphism) universal morphism doctrine (i.e.\ preserving universal quantifiers along $\mD$) such that the diagram
\[\xymatrix@+1pc{
& P^{\forall_{\mD}}\ar@{-->}[d]\\
P \ar[ru]^{\eta_P^{\forall_{\mD}}}\ar[r]_{\mathfrak{L}} & R
}\]
commutes.
\end{thmC}
\begin{remark}
    Notice that the previous free completions, in particular, the pure and the full existential completions, have been proved to be related to the so-called regular and exact completions of a category with finite limits \cite{SFEC,CARBONI199879}. We refer to \cite{maietti2023generalized} and \cite{maietti2021generalized} for a precise analysis and characterization of these results.  
\end{remark}

Now that we recalled both basic concepts of computability 
and the tools we need from Lawvere doctrines
we can start on the computability concepts we want to categorify here. First, we recall Medvedev reducibility and show it can be reformulated as a Medvedev doctrine.

\section{Medvedev doctrines}
\label{sec:medvedev}

The notion of Medvedev reducibility was introduced in the 50s, associated to a calculus of mathematical problems\footnote{Recently de Paiva and Da Silva showed that Kolmogorov problems can be seen as a variant of the Dialectica construction~\cite{depaiva2021Kolmogorov}.} in the style of Kolmogorov \cite{kolmogorov1932}, and now it is well established in the computability literature. We briefly introduce the main notions and definitions on the topic. For a more thorough presentation, the reader is referred to \cite{Sorbi1996,Hinman2012}.

A set $A\subseteq \Baire$ is sometimes called a \emph{mass problem}. The intuition is that a mass problem corresponds to the set of solutions for a specific computational problem. For example, the problem of deciding membership in a particular $P\subseteq \mathbb{N}$ corresponds to the mass problem $\{\chi_P \}$, where $\chi_P$ is the characteristic function of $P$. Similarly, the problem of enumerating $P$ corresponds to the family $\{ f\function{\mathbb{N}}{P} \st f \text{ is surjective}\}$. 

While Medvedev reducibility is usually defined in the context of Type-$2$ computability, we can give a slightly more general definition in the context of PCAs. 

\begin{definition}[Medvedev reducible set]
Let $\pca{A}$ be a PCA and let $\subpca{A}$ be an elementary sub-PCA of $\pca{A}$. If $A,B \subseteq \pca{A}$, we say that $A$ is \textbf{Medvedev reducible to} $B$, and write $A\medvedevreducible B$, if there is an effective functional $\Phi\in \subpca{A}$ such that $\Phi(B)\subseteq A$, i.e.\ $(\forall b\in B)(\Phi(b)\in A)$. 
\end{definition}
The notion of Medvedev reducibility induces a preorder on the powerset of $\pca{A}$, whose quotient is the \emph{Medvedev lattice}.
In the following, whenever there is no ambiguity, we identify a degree with any of its representatives.

\medskip

We will now introduce a \emph{Medvedev doctrine} that generalizes the notion of Medvedev reducibility. Intuitively, the Medvedev doctrine maps every singleton $X$ to an isomorphic copy of the Medvedev lattice. However, if $X$ is not a singleton, we obtain a somewhat different structure, corresponding to having several Medvedev reductions all witnessed by the same map. 

To define the Medvedev doctrine we first introduce the \emph{Turing doctrine}.
\begin{definition}
\thlabel{def:Medvedev doctrine of singletons}
Let $\pca{A}$ be a PCA and $\subpca{A}$ be an elementary sub-PCA. We can define a functor $\doctrine{\set}{\Mpcadoctrine}$ mapping a set $X$ to the set $\pca{A}^X$ of functions from $X$ to $\pca{A}$. Given two elements $\alpha,\beta\in \Mpcadoctrine(X)$, we define $\alpha\Mpcaorder \beta$ if there exists an element $\ovln{a}\in \subpca{A}$ such that for every $x\in X$ we have that $\ovln{a}\cdot \beta(x)$ is defined and $\ovln{a}\cdot \beta(x)=\alpha(x)$. The functor $\doctrine{\set}{\Mpcadoctrine}$ is called \textbf{Turing doctrine}.
\end{definition}

\begin{remark}
    \thlabel{remark:turing_doctrine}
    The name ``Turing doctrine'' is motivated by the fact that, when working with Kleene's second model, then $\Mpcadoctrine(1)$ can be identified with $\Baire$, and the reduction $\alpha \Mpcaorder \beta$ holds whenever there is a computable functional $\Phi$ such that $\Phi(\beta) = \alpha$, i.e.\ $(\forall n)(\Phi(\beta)(n)= \alpha(n))$. This corresponds precisely to the notion of Turing reducibility between functions $\mathbb{N}\to \mathbb{N}$.     
\end{remark}

Notice also that, in the previous example, it is important to consider Kleene's second model. Indeed, if we instead work with Kleene's first model, $\Mpcadoctrine(1)$ would be trivial (all numbers are computable) and $\Mpcadoctrine(\mathbb{N})$ would give rise to a stronger notion of reducibility than Turing reducibility: given $f,g\function{\mathbb{N}}{\mathbb{N}}$, $f\Mpcaorder g$ iff there is a computable functional $\Phi$ such that $(\forall n)(\Phi(g(n))=f(n))$, hence only the value of $g(n)$ is needed to compute $f(n)$. 

\begin{definition}[Medvedev doctrine] \thlabel{def:full medvedev doctrine}
    Given a PCA $\pca{A}$ with elementary sub-PCA $\subpca{A}$, we define the \textbf{Medvedev doctrine} $\doctrine{\set}{\MedvedevDoctrine}$ over $\set$ as follows: for every set $X$ and every pair of functions $\varphi,\psi$ in $\powerset(\pca{A})^X$, we define 
    \begin{align*}
        \varphi \MedvedevOrder \psi :\iff & (\exists \ovln{a}\in \subpca{A})(\forall x \in X)(\forall b\in \psi(x))(\exists a \in \varphi(x))(\ovln{a}\cdot b = a)\\
            \iff & (\exists \ovln{a}\in \subpca{A})(\forall x \in X)(\ovln{a}\cdot\psi(x) \subseteq \varphi(x)).
    \end{align*}
 This preorder induces an equivalence relation on functions in $\powerset(\pca{A})^X$. 
    The poset $\MedvedevDoctrine(X)$ is defined as the quotient of $\powerset(\pca{A})^X$ by the equivalence relation generated by $\MedvedevOrder$. The partial order on the equivalence classes $[\varphi]$ is the one induced by the Medvedev order $\MedvedevOrder$. Moreover, given a function $f\function{X}{Y}$, the functor $\MedvedevDoctrine_f\function{\MedvedevDoctrine(Y)}{\MedvedevDoctrine (X)}$ is defined as $\MedvedevDoctrine_f(\psi):=\psi \circ f$.
\end{definition}

Observe that, when working with Kleene's second PCA, if $X$ is a singleton then $\MedvedevDoctrine(X)$ corresponds exactly to the Medvedev degrees. For an arbitrary set $X$, the reduction $\varphi,\psi\in \MedvedevDoctrine(X)$ corresponds to a uniform Medvedev reducibility between $\varphi(x)$ and $\psi(x)$, where $x\in X$. In particular, if $X\subseteq \Baire$, we can think of $\varphi$ and $\psi$ as computational problems on the Baire space. In this case, the reduction $\varphi\MedvedevOrder \psi$ corresponds to a strong Weihrauch reduction where the forward functional is simply the identity function. We will discuss strong Weihrauch reducibility and its connections with the Medvedev doctrine in more detail in Sections~\ref{sec:strong_wei_doctrine}, \ref{sec:gen_weihrauch}.

We also show that, for every $X$, $\MedvedevDoctrine(X)$ is a distributive lattice with $\bot=\{\pca{A}\}$ and $\top=\emptyset$, where the join and the meet of the lattice are induced respectively by the following operation on mass problems: 
\begin{itemize}
    \item $A \lor B:= \{\pairing{a,b} \st a\in A \text{ and } b\in B\}$,
    \item $A \land B:= A \sqcup B = \{ \str{0}\concat a \st a \in A \} \cup \{\str{1}\concat b \st b\in B \}$, where $\str{n}\concat f(0):=n$ and $\str{n}\concat f(i+1):= f(i)$.
\end{itemize}

Next we want to show that, for every $X$, $\MedvedevDoctrine(X)$ is a co-Heyting algebra (i.e.\ a Brouwer algebra, see \cite[Thm.\ 9.1]{Sorbi1996}), where the subtraction operation is defined as $A\subtraction B := \min\{ C \st B \le A \lor B\}$. In other words, this subtraction is an `implication' with respect to the join of the lattice. The fact that, in general, $\MedvedevDoctrine(X)$ is not a Heyting algebra follows from the fact that $\MedvedevDoctrine(1)$ is not (see \cite[Thm.\ 9.2]{Sorbi1996}).

\begin{proposition}[Medvedev co-Heyting algebra]
\thlabel{thm:medvedev_algebra}
For every set $X$, $\MedvedevDoctrine(X)$ is a co-Heyting algebra, where:
\begin{enumerate}
    \item $\bot:=x\mapsto \pca{A}$
    \item $\top:=x\mapsto \emptyset$; 
    \item $(\varphi \wedge \psi)(x):=\{ \pairing{p_1, a} \st a \in \varphi(x)\}\cup \{ \pairing{p_2,b} \st b\in \psi(x)\}$, where $p_1, p_2$ are two fixed (different) elements in $\subpca{A}$. 
    \item $(\varphi \vee \psi)(x):=\{ \pairing{a,b}\st a\in \varphi (x) \text{ and } b\in \psi(x)\}$;
    \item $(\varphi\subtraction \psi )(x):=\{c\in \pca{A} \st (\forall b\in \psi(x))(c\cdot b\in \varphi(x))\}$.
\end{enumerate}
\end{proposition}
\begin{proof}
    This proposition can be proved essentially the same way one proves that the Medvedev degrees form a co-Heyting algebra (see \cite[Thm.\ 1.3]{Sorbi1996}). Let $\varphi,\psi \in \MedvedevDoctrine(X)$.
    \begin{enumerate}
        \item The reduction $\bot \MedvedevOrder \varphi$ is witnessed by the identity functional.
        \item The reduction $\varphi\MedvedevOrder \top$ is trivially witnessed by any $\ovln{a}\in\subpca{A}$, as the quantification on $b\in \top(x)$ is vacuously true.
        \item The reductions $\varphi \land \psi \MedvedevOrder \varphi$ and $\varphi \land \psi \MedvedevOrder \psi$ are witnessed respectively by the maps $a\mapsto \pairing{p_1, a}$ and $b\mapsto \pairing{p_2,b}$. Moreover, if $\rho\MedvedevOrder \varphi$ via $a_\varphi$ and $\rho\MedvedevOrder \psi$ via $a_\psi$ then the reduction $\rho \MedvedevOrder \varphi \land \psi$ is witnessed by the map that, upon input $\pairing{p,c}$, if $p=p_1$ returns $a_\varphi\cdot c$, otherwise returns $a_\psi \cdot c$.
        \item The reductions $\varphi \MedvedevOrder \varphi \lor \psi$ and $\psi \MedvedevOrder \varphi \lor \psi$ are witnessed by the projections. Moreover, if $\varphi \MedvedevOrder \rho$ via $a_\varphi$ and $\psi \MedvedevOrder \rho$ via $a_\psi$ then $\varphi \lor \psi \MedvedevOrder \rho$ is witnessed by the map $x\mapsto \pairing{a_\varphi\cdot x, a_\psi\cdot x}$.
        \item We need to show that, for every $\rho$, $\varphi\subtraction \psi \MedvedevOrder \rho \iff \varphi \MedvedevOrder \psi \lor \rho$. To prove the left-to-right implication, observe that if $\ovln{a}\in \subpca{A}$ witnesses the reduction $\varphi\subtraction \psi \MedvedevOrder \rho$, then, for every $\pairing{b,d}\in (\psi \lor \rho)(x)$, $\ovln{a}\cdot d\in(\varphi\subtraction \psi)(x)$, and therefore $(\ovln{a}\cdot d)\cdot b\in \varphi(x)$. To prove the right-to-left implication, notice that if $\ovln{b}$ witnesses $\varphi \MedvedevOrder \psi \lor \rho$, then, by definition, for every $\pairing{b,c}\in (\psi \lor \rho)(x)$, $\ovln{b}\cdot \pairing{b,c}\in \varphi(x)$. This implies that the map $b \mapsto \ovln{b}\cdot \pairing{b,c}\in (\varphi\subtraction \psi)(x)$, therefore concluding the proof. \qedhere
    \end{enumerate}
\end{proof}

Now we want to show some structural properties of Medvedev's doctrines.
Using the two theorems we recalled from previous work, we can show:

\begin{proposition}\thlabel{prop:structures of medvedev doctrine}
The Medvedev doctrine $\doctrine{\set}{\MedvedevDoctrine}$  is a full universal and pure existential doctrine. In particular, for every function $f\function{X}{Y}$, the morphism $\forall_f\function{\MedvedevDoctrine(X)}{\MedvedevDoctrine(Y)}$ sending an element $\varphi\in \MedvedevDoctrine(X)$ to the element $\forall_f (\varphi)\in \MedvedevDoctrine(Y)$ defined as
\[\forall_f (\varphi)(y):= \bigcup_{x\in f^{-1}(y)}\varphi(x)\]
is right adjoint to $\MedvedevDoctrine_f$, i.e.\ $\psi \MedvedevOrder \forall_f (\varphi)\iff \MedvedevDoctrine_f(\psi)\MedvedevOrder \varphi$ for any $\psi \in \MedvedevDoctrine(Y)$ and $\varphi\in \MedvedevDoctrine(X)$. 

Similarly, if $f$ is surjective then the assignment 
\[\exists_f (\varphi)(y):= \bigcap_{x\in f^{-1}(y)}\varphi(x)\]
determines a left adjoint to $\MedvedevDoctrine_f$, i.e.\ $\exists_f (\varphi)\MedvedevOrder \psi \iff \varphi \MedvedevOrder\MedvedevDoctrine_f(\psi)$ for any $\psi \in \MedvedevDoctrine(Y)$ and $\varphi\in \MedvedevDoctrine(X)$.
\end{proposition}
\begin{proof}
    This is essentially a definition-chasing exercise. Let us first show that, for every $\psi\in \MedvedevDoctrine(Y)$ and $\varphi\in\MedvedevDoctrine(X)$, $\psi \MedvedevOrder \forall_f (\varphi)\iff \MedvedevDoctrine_f(\psi)\MedvedevOrder \varphi$. Assume first that the reduction $\psi \MedvedevOrder \forall_f (\varphi)$ is witnessed by $\ovln{a}\in\subpca{A}$. The same $\ovln{a}$ witnesses $\MedvedevDoctrine_f(\psi)\MedvedevOrder \varphi$: indeed, for every $x\in X$ and every $a \in \varphi(x)$, we have $a \in \forall_f(\varphi)(f(x))$, and therefore $\ovln{a}\cdot a \in \psi(f(x)) = \MedvedevDoctrine_f(\psi)(x)$. On the other hand, assume $\ovln{b}$ witnesses $\MedvedevDoctrine_f(\psi)\MedvedevOrder \varphi$. For every $y\in Y$, if $b\in \forall_f(\varphi)(y)$ then $b\in \varphi(x)$ for some $x\in f^{-1}(y)$. In particular, $\ovln{b}\cdot b \in \MedvedevDoctrine_f(\psi)(x)= \psi(y)$, i.e.\ $\ovln{b}$ witnesses $\psi \MedvedevOrder \forall_f(\varphi)$.

    The second part of the statement is proved analogously. Assume first that $\exists_f (\varphi)\MedvedevOrder \psi$ is witnessed by $\ovln{c}$. Fix $x\in X$ and $a \in \MedvedevDoctrine_f(\psi)(x)=\psi(f(x))$. In particular, $\ovln{c}\cdot a \in \exists_f(\varphi)(f(x))$, and hence $\ovln{c}\cdot a \in \varphi(x)$. Finally, assume $\ovln{d}$ witnesses $\varphi \MedvedevOrder\MedvedevDoctrine_f(\psi)$, and let $y\in Y$ and $b\in \psi(y)$. Since $f$ is surjective, there is $x\in X$ s.t.\ $f(x)=y$. Moreover, for every $x\in f^{-1}(y)$ we have $b\in \psi(f(x))$ and, hence, $\ovln{d}\cdot b \in \varphi(x)$. This implies that  $\ovln{d}\cdot b \in \bigcap_{x\in f^{-1}(y)} \varphi(x) = \exists_f (\varphi)(y)$.
\end{proof}

The rest of this section is devoted to studying abstract \textbf{universal} properties of Medvedev doctrines.

\begin{theorem}[Medvedev isomorphism]
\thlabel{thm:medvedev_iso}
Let $\doctrine{\set}{\MedvedevDoctrine}$ be the Medvedev doctrine for a given PCA $\pca{A}$. Then we have the isomorphism of full universal doctrines
\[\MedvedevDoctrine\equiv \compunfull{\Mpcadoctrine}.\]
In particular, for every set $X$, the map 
\[(f,\alpha)\mapsto \alpha\circ f^{-1} \]
is a surjective (pre)order homomorphism between $(\compunfull{\Mpcadoctrine}(X),\lefullunicomp)$ and $(\MedvedevDoctrine(X),\MedvedevOrder)$.
\end{theorem}

\begin{proof}
    Observe that the full universal completion $\compunfull{\Mpcadoctrine}$ of $\Mpcadoctrine$ can be described as follows:
\begin{itemize}
    \item $\compunfull{\Mpcadoctrine}(X) = \{ (f,\alpha) \st f\function{Y}{X} \text{ and } \alpha \in \Mpcadoctrine(Y) \},$
    \item if $f\function{Y}{X}$, $\alpha\in\Mpcadoctrine(Y)$, $g\function{Z}{X}$, and $\beta\in \Mpcadoctrine(Z)$, then 
    \[ (f,\alpha)\lefullunicomp (g,\beta) \iff (\exists h\function{Z}{Y})((\forall z \in Z)(g(z) = f\circ h(z)) \text{ and } (\Mpcadoctrine)_h(\alpha) \Mpcaorder \beta).\] 
\end{itemize}
In other words, $(f,\alpha)\lefullunicomp (g,\beta)$ if and only if there are $h\function{Z}{Y}$ and $\ovln{a}\in\pca{A}$ such that the following diagram commutes 
\begin{center}
\begin{tikzcd}
Y \arrow[rd, "f"] \arrow[dd, "\alpha"'] &   & Z \arrow[ld, "g"'] \arrow[ll, "h"'] \arrow[dd, "\beta"] \\
                                                       & X &                                                                    \\
    \pca{A}                             &   & \pca{A}  \arrow[ll, "\ovln{a}"]          
    \end{tikzcd}    
\end{center}

If we define $\varphi:= \alpha\circ f^{-1}$ and $\psi:=\beta\circ g^{-1}$. we obtain the following diagram:
\begin{center}
    \begin{tikzcd}
    Y \arrow[rd, "f"] \arrow[dd, "\alpha"'] &                                                                    & Z \arrow[ld, "g"'] \arrow[ll, "h"'] \arrow[dd, "\beta"] \\
                                            & X \arrow[ld, "\varphi", Rightarrow] \arrow[rd, "\psi"', Rightarrow] &                                                         \\
    \pca{A}                  &                                                                    & \pca{A}\arrow[ll, "\ovln{a}"]                                         
    \end{tikzcd}    
\end{center}
where $\varphi$ and $\psi$ are represented as double arrows to stress the fact that they are maps into $\powerset(\pca{A})$.

Let us first show that $(f,\alpha)\lefullunicomp (g,\beta)$ implies $\varphi \MedvedevOrder \psi$. Fix two witnesses $h\function{Z}{Y}$ and $\ovln{a}\in \subpca{A}$ for $(f,\alpha)\lefullunicomp (g,\beta)$ and fix $x\in X$. If $x\notin \ran(g)$ then $\psi(x)=\emptyset$, hence there is nothing to prove. Assume therefore that $x\in \ran(g)$. By definition, $\psi(x)=\beta(g^{-1}(x)) \neq\emptyset$. Fix $b\in \psi(x)$ and let $z\in g^{-1}(x)$ be s.t.\ $\beta(z)=b$. Since $(f,\alpha)\lefullunicomp (g,\beta)$, we can write 
\[  \ovln{a}\cdot b = \ovln{a}\cdot\beta(z)=\alpha\circ h(z) = \alpha(y), \]
for some $y\in f^{-1}(x)$. In particular, $\alpha(y)\in \varphi(x)$, and therefore $\ovln{a}$ witnesses $\varphi \MedvedevOrder \psi$.

Let us now prove the other direction. Let $\ovln{b}\in \subpca{A}$ be a witness for $\varphi \MedvedevOrder \psi$. Fix $z\in Z$. Clearly, letting $x_z:=g(z)$, $\psi(x_z)\neq \emptyset$. Moreover, $\psi(x_z)\neq \emptyset$ implies $\varphi(x_z)=\alpha(f^{-1}(x_z))\neq\emptyset$ (as $\ovln{b}\cdot \psi(x_z) \subseteq \varphi(x_z)$).  
We define $h$ as a choice function that maps $z$ to some $y\in f^{-1}(x_z)$ such that $\alpha(y)\in \ovln{b}\cdot \psi(x_z)$. Observe that $h$ is well-defined as if $a \in \ovln{b} \cdot \psi(x_z)$ then $a =\alpha(y)$ for some $y\in f^{-1}(x_z)$. This also shows that $h$ and $\ovln{b}$ witness $(f,\alpha)\lefullunicomp (g,\beta)$.

Finally, to show that the homomorphism is surjective it is enough to notice that every $\varphi\function{X}{\powerset(\pca{A})}$ is the image of the pair $(\pi_X,\pi_{\pca{A}})$, where $Y:= \{ (x,a) \in X \times \pca{A} \st a\in\varphi(x) \}$, and $\pi_X\function{Y}{X}$ and $\pi_{\pca{A}}\function{Y}{\pca{A}}$ are the two projections.
\end{proof}

In other words, when working with Kleene's second model, the Medvedev degrees are isomorphic to the full universal completion of the Turing doctrine. 

\begin{remark}
Notice that, when we consider a PCA $\pca{A}$ and the trivial elementary sub-PCA given by $\pca{A}$ itself, we have that the Medvedev doctrine $\doctrine{\set}{\MedvedevDoctrine}$ can be presented as the functor obtained by composing the realizability doctrine $\doctrine{\set}{\mathcal{R}}$  defined in \thref{ex: realizability tripos} with the op-functor $(-)^{\op}:\pos\to\pos$ inverting the order of posets.

Since it is known that the fibres of the realizability doctrine have the Heyting structure, this presentation provides, for example, an abstract explanation of the co-Heyting structure of the fibres of Medvedev doctrines presented in \thref{thm:medvedev_algebra}. 

Moreover, as it happens for the realizability triposes, notice that the Medvedev doctrine $\doctrine{\set}{\MedvedevDoctrine}$ can be equivalently presented in a second equivalent way via the well-known embedding (called \emph{constant objects functor}) $\nabla:\set \to \mathsf{RT}(\pca{A}',\pca{A})$ of the category of sets into a (relative) realizability topos (see e.g.\ \cite{van_Oosten_realizability}). We also refer to \cite{frey_streicher_2021} for a careful analysis and generalization of this approach to triposes. In particular, the doctrine $\MedvedevDoctrine$ happens to be equivalent to the functor mapping a given set $X$ into the poset given by $\mathsf{Sub}(\nabla(X))^{\op}$. 

In the specific case of the \emph{Kleene-Vesley topos} $\mathsf{RT}(\mathcal{K}_2^{rec},\mathcal{K}_2)$ (see for example \cite[p.\ 279]{van_Oosten_realizability} and the references therein) the fact that the poset given by $\mathsf{Sub}(\nabla(1))^{\op}$ provides exactly the Medvedev lattice has been already observed in \cite[p.\ 280]{van_Oosten_realizability}.
\end{remark}

\medskip

\subsection{Muchnik Doctrines}
Using a very similar strategy, we now show that the \textbf{Muchnik lattice} is isomorphic to the full universal completion of a doctrine. The notion of \emph{Muchnik reducibility} formally dates back to 1963 (but it was probably known earlier) and can be thought of as the non-uniform version of Medvedev reducibility. More precisely, given two mass problems $P,Q \subseteq \Baire$, we say that $P$ is Muchnik reducible to $Q$, and write $P\le_w Q$, if for every $q\in Q$ there is a functional $\Phi$ such that $\Phi(q) \in P$. In other words, $P\le_w Q$ if every element of $Q$ computes some element of $P$. Muchnik reducibility is sometimes called ``weak reducibility" (which motivates the choice of the symbol $\le_w$) to contrast it with Medvedev reducibility, sometimes called ``strong reducibility" \cite{Hinman2012}.

Similarly to the Medvedev degrees, the Muchnik degrees form a distributive lattice, where the join and the meet operations are induced by the same operations on subsets of $\Baire$ that induce the join and the meet in the Medvedev degrees. Unlike the Medvedev lattice, the Muchnik lattice is both a Heyting and a co-Heyting algebra \cite[Prop.\ 4.3 and 4.7]{Hinman2012}.

\begin{definition}[Muchnik doctrine of singletons]
\thlabel{def:Muchnik doctrine of singletons}
    Let $\pca{A}$ be a PCA and let $\subpca{A}$ be an elementary sub-PCA. We can define a doctrine $\doctrine{\set}{\Mucpcadoctrine}$ mapping a set $X$ to the set $\pca{A}^X$ of functions from $X$ to $\pca{A}$. Given two elements $\alpha,\beta\in \Mucpcadoctrine(X)$, we define $\alpha\Mucpcaorder \beta$ if 
    \[ (\forall x\in X)(\exists \ovln{a}\in \subpca{A})(\ovln{a}\cdot \beta(x) = \alpha(x)). \]
    The functor $\doctrine{\set}{\Mucpcadoctrine}$ is called \textbf{Muchnik doctrine of singletons}.   
\end{definition}

Observe that the order $\Mucpcaorder$ is just the non-uniform version of the order $\Mpcaorder$ induced by the Turing doctrine. In particular, if $X$ is a singleton then $\Mucpcadoctrine(X)= \Mpcadoctrine(X)$. This implies that Remark~\ref{remark:turing_doctrine} applies to the Muchnik doctrine of singletons as well. On the other hand, both $\Mucpcadoctrine(1)$ and $\Mucpcadoctrine(\mathbb{N})$ are trivial when working with Kleene's first model. 

\begin{definition}[Muchnik doctrine]
\thlabel{def:Muchnik doctrine}
      We define the \textbf{Muchnik doctrine} $\doctrine{\set}{\MuchnikDoctrine}$ over $\set$ as follows. For every set $X$ and every pair of functions $\varphi,\psi \in \powerset(\pca{A})^X$, we define 
    \[ \varphi \MuchnikOrder \psi :\iff (\forall x \in X)(\forall b\in \psi(x))(\exists \ovln{a}\in \subpca{A})(\ovln{a}\cdot b \in \varphi(x)).  \]
    This preorder induces an equivalence relation on functions in $\powerset(\pca{A})^X$. The doctrine $\MuchnikDoctrine(X)$ is defined as the quotient of $\powerset(\pca{A})^X$ by the equivalence relation generated by $\MuchnikOrder$. The partial order on the equivalence classes $[\varphi]$ is the one induced by $\MuchnikOrder$. 
\end{definition}

\begin{proposition}
For every set $X$, $\MuchnikDoctrine(X)$ is both a Heyting and a co-Heyting algebra, where:
\begin{enumerate}
    \item $\bot:=x\mapsto \pca{A}$ 
    \item $\top:=x\mapsto \emptyset$; 
    \item $(\varphi \wedge \psi)(x):=\{ \pairing{p_1, a} \st a \in \varphi(x)\}\cup \{ \pairing{p_2,b} \st b\in \psi(x)\}$, where $p_1, p_2$ are two fixed (different) elements in $\subpca{A}$. 
    \item $(\varphi \vee \psi)(x):=\{ \pairing{a,b}\st a\in \varphi (x) \text{ and } b\in \psi(x)\}$;
    \item $(\varphi\subtraction \psi )(x):=\{c\in \pca{A} \st (\forall b\in \psi(x))(c\cdot b\in \varphi(x))\}$.
    \item $(\varphi\rightarrow \psi )(x):=\{b\in \psi(x) \st (\forall \ovln{a} \in \subpca{A})(\ovln{a}\cdot b \notin \varphi(x))\}$.
\end{enumerate}
\end{proposition}
\begin{proof}
    The points (1)-(5) can be proved as in the proof of \thref{thm:medvedev_algebra}, so we only prove point (6). The argument is a straightforward generalization of \cite[Prop.\ 4.3]{Hinman2012}. We need to show that, for every $\rho$, $\rho \MuchnikOrder \varphi \rightarrow \psi \iff \varphi \wedge \rho \MuchnikOrder \psi$. Observe that,
    \begin{align*}
        \rho \MuchnikOrder \varphi\rightarrow \psi & \iff (\forall x \in X)(\forall b \in (\varphi\rightarrow \psi)(x))(\exists \ovln{a}\in\subpca{A})(\ovln{a}\cdot b \in \rho(x)) \\
            & \iff (\forall x \in X)(\forall b \in \psi(x))((\forall \ovln{a}\in\subpca{A})(\ovln{a}\cdot b \notin\varphi(x)) \Rightarrow (\exists \ovln{b}\in\subpca{A})(\ovln{b}\cdot b \in \rho(x)) ) \\
            & \iff (\forall x \in X)(\forall b \in \psi(x))((\exists \ovln{a}\in\subpca{A})(\ovln{a}\cdot b \in\varphi(x)) \lor (\exists \ovln{b}\in\subpca{A})(\ovln{b}\cdot b \in \rho(x)) ) \\
            & \iff (\forall x \in X)(\forall b \in \psi(x))(\exists \ovln{c}\in\subpca{A})( (\exists a \in \varphi(x))(\ovln{c}\cdot b = \pairing{p_1,a}) \\
            & \mathrel{\phantom{\iff}} \lor (\exists d \in \rho(x))(\ovln{c}\cdot b = \pairing{p_2,d})  )\\
            & \iff (\forall x \in X)(\forall b \in \psi(x))(\exists \ovln{c}\in\subpca{A})( \ovln{c}\cdot b \in \varphi\wedge \rho(x)) \\
            & \iff \varphi \wedge \rho \MuchnikOrder \psi.\tag*{\qedhere}
    \end{align*}
\end{proof}

It is straightforward to adapt the proof of \thref{thm:medvedev_iso} to show the following:

\begin{theorem}[Muchnik isomorphism]
    The Muchnik doctrine $\MuchnikDoctrine$ is isomorphic to the full universal completion $\compunfull{(\Mucpcadoctrine)}$.
\end{theorem}
Observe that the Turing doctrine $\doctrine{\set}{\Mpcadoctrine}$ can be embedded in the Muchnik doctrine of singletons $\doctrine{\set}{\Mucpcadoctrine}$. Let us denote such a morphism of doctrines by $\arrow{\Mpcadoctrine}{ \mathfrak{I} }{\Mucpcadoctrine}$. 
By \thref{theorem full universal comp} we can employ the universal property of the full universal completion and conclude the following corollary.
\begin{corollary}
Let $\pca{A}$ be a given PCA and let $\pca{A}'$ be an elementary sub-PCA.
Then there exists a morphism of full universal doctrines such that the diagram
\[\xymatrix@+1pc{
& &\MedvedevDoctrine \ar@{-->}[d]\\
\Mpcadoctrine \ar[rru]^{\eta_{\Mpcadoctrine}^{\forall_{\mathsf{f}}}}\ar[r]_{\mathfrak{I}} & \Mucpcadoctrine \ar[r]_{\eta_{\Mucpcadoctrine}^{\forall_{\mathsf{f}}}} & \MuchnikDoctrine
}\]
commutes.
\end{corollary}
This corollary is the categorical version of how Medvedev's (strong) reducibility is the uniform version of Muchnik's (weak) reducibility.
\begin{remark}
    Notice that the results and considerations presented for Muchnik doctrines could be presented in a more general setting. In particular, if we consider doctrines with pointwise order, e.g.\ Muchnik doctrines and localic doctrines, the full universal completion of these doctrines is given precisely by the doctrine of their pointwise domination order (Smyth order~\cite{SMYTH197823}). Intuitively, the full universal completion can be seen as a generalization of the construction of the Smyth preorder on the powerset of a given preorder.
\end{remark}

We conclude this section by explicitly stating the following theorem that collects some already mentioned facts.

\begin{theorem}
    The Turing degrees, Medvedev degrees, and Muchnik degrees are isomorphic respectively to $\Mpcadoctrine(1)$, $\MedvedevDoctrine(1)$, and $\MuchnikDoctrine(1)$. 
\end{theorem}

\section{Weihrauch Doctrines}
\label{sec:weihrauch_doctrine}

Weihrauch reducibility is a notion of reducibility between computational problems that is useful to calibrate the uniform computational strength of a multi-valued function. It complements the analysis of mathematical theorems done in reverse mathematics, as multi-valued functions on represented spaces can be considered as realizers of theorems in a natural way. Weihrauch reducibility provides a framework where one can formalize questions such as ``which theorems can be transformed continuously or computably into another?'' \cite{Brattka_2011}. This provides a purely topological or computational approach to metamathematics that sheds new light on the nature of theorems. Given the connection between Medvedev and Weihrauch reducibility, it is natural to ask whether, with a similar approach as the one from the last section, we can also describe Weihrauch reducibility as the completion of some existential doctrine. 

We start by recalling the main objects of interest of Weihrauch reducibility, and by fixing the notation for the rest of this section:
\begin{definition}
A \textbf{(partial) multi-valued function} from a set $X$ to a set $Y$, written as $f\pmfunction{X}{Y}$, is a function $f\function{X}{\powerset(Y)}$ into the powerset of $Y$. The domain of $f$ is the set $\dom(f)=\{ x\in X \st f(x) \neq \emptyset\}$. Whenever $f(x)$ is a singleton for every $x\in\dom(f)$, we write $f(x)=y$ instead of $f(x)=\{y\}$. 
\end{definition}
It helps the intuition to think of multi-valued functions as computational problems, namely instance-solution pairs, where a single problem instance can have multiple solutions.

While Weihrauch reducibility is often introduced in the context of Type-$2$ computability, we will introduce it in the more general context of PCAs, as we did for the Medvedev reducibility relation.
\begin{definition}[Weihrauch reducibilty]
    \thlabel{def:wei_classic}
  Let $\pca{A}$ be a PCA, $\pca{A}'$ be an elementary sub-PCA of $\pca{A}$ and 
  let $f,g\pmfunction{\pca{A}}{\pca{A}}$ be partial multi-valued functions on $\pca{A}$. We say that $f$ is \textbf{Weihrauch reducible} to $g$, and write $f\weireducible g$, if there are $\ovln h, \ovln k\in \subpca{A}$ such that
   \[ (\forall G\vdash g)( \ovln{h}(\id, G\ovln{k})\vdash f ), \] 
   where $\ovln{h}(\id, G\ovln{k}) := p \mapsto \ovln{h} \cdot \pairing{p, G(\ovln{k}\cdot p)}$ and $G \vdash g$ means that $G\pfunction{\pca{A}}{\pca{A}}$ is such that, for every $p\in \dom(g)$, $G(p)\in g(p)$. This corresponds to saying that $G$ is a \emph{realizer}\footnote{The notion of realizers will be introduced more generally in \ref{def:rep_spaces_realizer} in the context of (multi-)represented spaces. The existence of a realizer for every multi-valued function depends on a (relatively) weak form of the axiom of choice.} of $g$. The functionals $\ovln k, \ovln h$ are sometimes referred to as \emph{forward} and \emph{backward} functional respectively.
\end{definition}

To properly introduce a Lawvere doctrine abstracting the notion of Weihrauch reducibility into the categorical setting, we start by defining a category that  plays the role of the base category of the doctrine. 
\begin{definition}
    Let $\pca{A}$ be a PCA, and let $\pca{A}'$ be an elementary sub-PCA of $\pca{A}$. We say that a function $f\function{X}{Y}$ is $\pca{A}'$\textbf{-computable}\footnote{This notation can be found in \cite[Sec.\ 1]{BIRKEDAL2002115}.} if there exists an element $a\in \pca{A}'$ such that for every element $x\in X$, $a\cdot x\downarrow$ and $f(x)=a\cdot x$. We avoid explicitly mentioning $\subpca{A}$ and simply write \emph{computable} whenever the elementary sub-PCA is clear from the context.
\end{definition}

Throughout this section, with a small abuse of notation, we will often replace a computable function $f$ with the element of the elementary sup-PCA that realizes it.

Observe that, since $\subpca{A}$ is an \emph{elementary} sub-PCA of $\pca{A}$, we have that the identity function is computable, and that the composition of computable functions is well-behaved. Hence, we can define the following category:
\begin{definition}\thlabel{def:PCA_cat}
    Let $\pca{A}$ be a PCA, and let $\pca{A}'$ be an elementary sub-PCA of $\pca{A}$. We denote\footnote{The choice of the notation $\PcaCat$ is motivated by the fact that, as we will see later in \thref{rem:partition_modests_is_PcaCat}, $\PcaCat$ is precisely the full subcategory of the category of modest sets which are partitioned.} by $\PcaCat (\pca{A},\subpca{A})$ the category whose objects are subsets $X\subseteq \pca{A}$. A morphism $f\in \PcaCat(\pca{A},\subpca{A})(X,Y)$ is a computable function $f\function{X}{Y}$.
\end{definition}

\begin{remark}
    Notice that the category $\PcaCat(\pca{A},\subpca{A})$ has finite products. This follows from the fact that $\pca{A}'$ is an \emph{elementary} sub-PCA of $\pca{A}$, and hence all the standard operations of pairing and projections, which allows us to define the categorical products, are computable morphisms. In particular, the categorical product between $X,Y\in\PcaCat(\pca{A},\subpca{A})$ is given by the set $X\times Y := \{ \pairing{x,y} \st x\in X \text{ and } y \in Y\}$ and the projections are the maps $\fprAsm, \sprAsm$ defined after \thref{def: PCA}. We also denote with $\ovln{1}$ the terminal object of $\PcaCat(\pca{A},\subpca{A})$, defined as $\ovln{1}:=\{ k \}$, where $k$ is the identity functional (\thref{def: PCA}).
\end{remark}

For our purposes, it is convenient to rewrite the ordinary definition of Weihrauch reducibility as follows:

\begin{definition}[Weihrauch reducibility]
\thlabel{def:wei_relation}
Let $X,Z$ be two object of $\PcaCat(\pca{A},\subpca{A})$. Given two functions $f\function{X}{\powerset^*(\pca{A})}$ and $g\function{Z}{\powerset^*(\pca{A})}$, $f\weireducible g$ if and only if there is a computable function $\ovln{k}\function{X}{Z}$ and an element $\ovln{h}\in\subpca{A}$ such that
\[ (\forall p\in X)(\forall q\in g(\ovln{k}\cdot p))(\ovln{h}\cdot \pairing{p,q}\in f(p)) . \] 
\end{definition}

It is well-known that there is a close connection between Weihrauch and Medvedev reducibility. Indeed, if $f\weireducible g$ via the functionals $\ovln{k}$ and $\ovln{h}$, then $\ovln{k}$ witnesses the reduction $\dom(g)\medvedevreducible \dom(f)$. Besides, for every $p\in \dom(f)$ there is a uniform continuous Medvedev reduction between $f(p)$ and $g(\ovln{k}\cdot p)$. Indeed, the map $q\mapsto \ovln{h}\cdot \pairing{p, q}$ maps every element of $g(\ovln{k}\cdot p)$ to some element of $f(p)$. If the map $\ovln{h}$ need not receive $p$ as part of its input, then $\ovln{h}$ would witness a uniform (computable) Medvedev reducibility between $f(p)$ and $g(\ovln{k}\cdot p)$. This stronger requirement leads to the notion of \emph{strong Weihrauch reducibility}, which will be studied more in detail in the next section.

Now we generalize the notion of Weihrauch reducibility similarly to what we did for the Medvedev reducibility. To this end, we introduce the following doctrine.

\begin{definition}[elementary Weihrauch doctrine]
\thlabel{def:Weihrauch doctrine of singletons}
    We define the \textbf{elementary Weihrauch doctrine} $\doctrine{\PcaCat(\pca{A},\subpca{A})}{\WeiElementary}$ as follows: for every object $X$ of $\PcaCat(\pca{A},\subpca{A})$, the objects $\WeiElementary(X)$ are functions $f\colon X\to  \powerset^*(\pca{A})$. For every pair of maps $f,g$ in $\powerset^*(\pca{A})^X$, we define $f \WeiOrderElementary g$ iff there is $\ovln{h}\in \subpca{A}$ such that 
    \[ (\forall p\in X)(\forall q\in g(p))(\ovln{h}\cdot \pairing{p,q}\in f(p)). \]

    This preorder induces an equivalence relation on functions in $\powerset^*(\pca{A})^X$. The doctrine $\WeiElementary(X)$ is defined as the quotient of $\powerset^*(\pca{A})^X$ by the equivalence relation generated by $\WeiOrderElementary$. The partial order on the equivalence class $[f]$ is the one induced by $\WeiOrderElementary$. The action of $\WeiElementary$ on the morphisms of $\PcaCat(\pca{A},\subpca{A})$ is defined by the pre-composition.
\end{definition}

Notice that a reduction $f \WeiOrderElementary g$ corresponds to a Weihrauch reducibility between $f$ and $g$ defined on the same set $X$, where the forward functional is witnessed by the identity functional. 

Observe that the doctrine $\WeiElementary$ is well-defined as a functor between $\PcaCat(\pca{A},\subpca{A})^{\mathrm{op}}$ and $\mathsf{Pos}$. Indeed, it is straightforward to check that, for every $X,Z \in\PcaCat(\pca{A},\subpca{A})$, every $f,g \in \WeiElementary(X)$, and every morphism $\ovln{k}\in \PcaCat(\pca{A},\subpca{A})(Z,X)$, any function $\ovln{h}\in\subpca{A}$ witnessing $f\WeiOrderElementary g$ is a witness for $f\circ \ovln{k} = \WeiElementary_{\ovln{k}}(f) \WeiOrderElementary \WeiElementary_{\ovln{k}}(g) = g\circ \ovln{k}$. While Weihrauch reductions are not, in general, preserved under pre-composition, the reduction $f \WeiOrderElementary g$ essentially establishes a ``layerwise'' relation between $f$ and $g$, namely to solve $f(x)$ we only need to look at $g(x)$. It is therefore apparent that this relation is not affected by pre-compositions.

\begin{proposition}\thlabel{prop:structures of elementary weihrauch doctrine}
The elementary Weihrauch doctrine $\doctrine{\PcaCat(\pca{A},\subpca{A})}{\WeiElementary}$ is a pure universal doctrine. In particular, for every product projection $\pi_Z\function{X\times Z}{Z}$, the morphism $\forall_{\pi_Z}\function{\WeiElementary(X\times Z)}{\WeiElementary(Z)}$ sending an element $f\in \WeiElementary(X\times Z)$ to the element $\forall_{\pi_Z} (f)\in \WeiElementary(Z)$ defined as
\[\forall_{\pi_Z}  (f)(z):= \bigcup_{x\in X} \left\{ \pairing{x,y} \st y \in f(x,z)  \right\}\]
is right-adjoint to $\WeiElementary_{\pi_Z}(g)$, i.e.\ for every $f\in \WeiElementary(X\times Z)$ and $g \in \WeiElementary(Z)$,
\[g \WeiOrderElementary \forall_{\pi_Z}(f)\iff \WeiElementary_{\pi_Z} (g)\WeiOrderElementary f.\]  
\end{proposition}
\begin{proof}

    Assume first that $g \WeiOrderElementary \forall_{\pi_Z}(f)$ via $\ovln{a}$. Let $\ovln{b}\in\subpca{A}$ be defined as $\ovln{b}\cdot \pairing{\pairing{t,p},s} := \ovln{a}\cdot\pairing{p,\pairing{t,s}}$. We want to show that for every $(x,z)\in X\times Z$, 
    \[ (\forall y \in f(x,z))(\ovln{b}\cdot \pairing{\pairing{x,z},y} \in g\circ \pi_Z(x,z) = g(z)).\] 
    To this end, fix $(x,z)\in X\times Z$. Since $f$ is total on $X\times Z$, we only need to show that for every $y\in f(x,z)$, $\ovln{b}\cdot \pairing{\pairing{x,z},y} \in g(z)$. This follows immediately from the definition of $\ovln{b}$, as $g \WeiOrderElementary \forall_{\pi_Z}(f)$ and $\pairing{x,y}\in \forall_{\pi_Z}(f)(z)$.

    For the right-to-left direction, we proceed analogously: assume $\WeiElementary_{\pi_Z} (g)\WeiOrderElementary f$ via $\ovln{c}$. Define $\ovln{d}\cdot \pairing{p,\pairing{t,s}}:=\ovln{c}\cdot\pairing{\pairing{t,p},s}$. To show that $g \WeiOrderElementary \forall_{\pi_Z}(f)$, simply notice that for every $z\in Z$, and every $\pairing{x,y}\in \forall_{\pi_Z}(f)(z)$, the fact that $\WeiElementary_{\pi_Z} (g)\WeiOrderElementary f$ via $\ovln{c}$ implies that $\ovln{d}\cdot \pairing{z,\pairing{x,y}}=\ovln{c}\cdot\pairing{\pairing{x,z},y} \in \WeiElementary_{\pi_Z} (g)(x,z)= g(z)$.
\end{proof}

\medskip

Now we introduce the Weihrauch doctrine. To this end, we introduce first the notion of \emph{generalized Weihrauch predicate}.

\begin{definition}[Weihrauch predicate]
    A \textbf{generalized Weihrauch predicate} on $X\subseteq \pca{A}$ is a function
    \[F\function{X}{\powerset^*(\pca{A})^Y}\]
    for some $Y\subseteq \pca{A}$.
\end{definition}

Given a generalized Weihrauch predicate $F$, we denote by $F_x\function{Y}{\powerset^* (\pca{A})}$ the function $F(x)$. Moreover, with a small abuse of notation, we can write $F(x,y)$ for $F(x)(y)$.

\begin{definition}[Weihrauch doctrine]\thlabel{def:Weihrauch doctrine}
    Given a PCA $\pca{A}$ with elementary sub-PCA $\subpca{A}$, the \textbf{Weihrauch doctrine} is the functor $\doctrine{\PcaCat(\pca{A},\subpca{A})}{\WeiDoctrine}$ that maps $X\subseteq \pca{A}$ to the preorder $\WeiDoctrine(X)$ defined as follows:
    \begin{itemize}
        \item objects are generalized Weihrauch predicates on $X$;
        \item the partial order is given by the poset reflection of the preorder defined as follows: let $Y,Z\subseteq \pca{A}$. For every $F\function{X}{\powerset^*(\pca{A})^Y}$ and $G\function{X}{\powerset^*(\pca{A})^Z}$, we say that $F\WeiDoctrineOrder G$ if there exist a computable function $ \ovln{k}\function{X \times Y}{Z} $ and $\ovln{h}\in \subpca{A}$ such that
        \[ (\forall \pairing{x,y}\in X\times Y)(\forall q\in G(x,\ovln{k}\cdot \pairing{x,y}))(\ovln{h}\cdot \pairing{\pairing{x,y},q}\in F(x,y)).\]
        \end{itemize}
\end{definition}

Intuitively, if $F$ and $G$ are generalized Weihrauch predicates, the reduction $F \WeiDoctrineOrder G$ can be seen as a uniform sequence of Weihrauch reductions, one for each $x\in X$, all witnessed by the same reduction functionals. Observe that the doctrine is well-defined as pre-composition with a morphism $\ovln{j}\function{W}{X}$ preserves the reduction $F\WeiDoctrineOrder G$.
    
The following result is immediate.
\begin{theorem}[Weihrauch lattice]
    \thlabel{thm:weihrauch_lattice}
    The Weihrauch lattice is isomorphic to $\WeiDoctrine(\ovln{1})$.
\end{theorem}
\begin{proof}
    By definition of Weihrauch doctrine, the objects of the poset $\WeiDoctrine(\ovln{1})$ can be identified with functions $Y\to \powerset^*(\pca{A})$. Moreover, if $F\function{Y}{\powerset^*(\pca{A})}$ and $G\function{Z}{\powerset^*(\pca{A})}$ are in $\WeiDoctrine(\ovln{1})$ then $F\WeiDoctrineOrder G$ iff there are a computable function $\ovln{k}\function{Y}{Z}$ and $\ovln{h}\in\subpca{A}$ such that
     \[ (\forall y\in Y)(\forall q\in G(\ovln{k}\cdot y))(\ovln{h}\cdot \pairing{y,q}\in F(y)), \]
    which corresponds to \thref{def:wei_relation}.
\end{proof}

In particular, when working with Kleene's second model, $\WeiDoctrine(\ovln{1})$ yields precisely the classical Weihrauch degrees.

Now we show that the Weihrauch doctrine is isomorphic to the pure existential completion of the elementary Weihrauch doctrine. Recall that, by definition, the elements of $\compex{\WeiElementary}(X)$ are pairs of the form $(\pi,f)$, where $\pi\function{X\times Y}{X}$ is the projection on $X$ and $f\in \WeiElementary(X\times Y)$, i.e.\ $f\function{X\times Y}{\powerset^*{\pca(A)}}$. Moreover, 
\[ (\pi  ,f) \leexcomp (\pi',g) \iff (\exists \ovln{k}\in \PcaCat(\pca{A},\subpca{A})(X\times Y,Z))(f \WeiOrderElementary \WeiElementary_{\angbr{\pi_X}{k}}(g)), \]
where $\WeiElementary_{\angbr{\pi_X}{k}}(g) \colon= (x,y)\mapsto g(x,\ovln{k}(x,y))$.

\begin{theorem}[Weihrauch isomorphism]\thlabel{thm:Weihrauch isomorphism}
    The Weihrauch doctrine is isomorphic to the pure existential completion of the elementary Weihrauch doctrine. In symbols:
    \[\WeiDoctrine\equiv \compex{\WeiElementary}.\]
    In particular, fixed a PCA $\pca{A}$ with elementary sub-PCA $\subpca{A}$, for every $X\subseteq \pca{A}$ the map 
    \[(\pi_X,f)\mapsto F \]
    where $F\function{X}{\powerset^*(\pca{A})^Y}$ is defined as $F(x):=f(x,\cdot)$, is an isomorphism of posets  between $(\compex{\WeiElementary}(X),\leq_{\exists})$ and $(\WeiDoctrine(X),\WeiDoctrineOrder)$.
\end{theorem}
\begin{proof}
    We first show that the assignment $(\pi_X,f)\mapsto F$ preserves and reflects the order. Let $(\pi ,f)$ and $(\pi',g)$ be two elements of $\compex{\WeiElementary}(X)$. By expanding the definition of $\WeiOrderElementary$ and writing $g(x,k(x,y))$ in place of $ \WeiElementary_{\angbr{\pi_X}{k}}(g)(x,y)$, we obtain that $(\pi  ,f) \leexcomp (\pi',g) $ iff there is a morphism $\ovln{k}\function{X\times Y}{Z}$ and $\ovln{h}\in\subpca{A}$ such that 
    \[ (\forall (x,y) \in X\times Y)(\forall q \in g(x,\ovln{k}\cdot\pairing{x,y}) )(\ovln{h}\cdot \pairing{\pairing{x,y},q} \in f(x,y))\tag{$\star$}\] 
    Let $F\function{X}{\powerset^*(\pca{A})^Y}$ and $G\function{X}{\powerset^* (\pca{A})^Z}$ be the images of $(\pi,f)$ and $(\pi',g)$ respectively. Observe that, writing $F(x,y)=f(x,y)$ (analogously for $G$) and substituting $F$ and $G$ in ($\star$), it is straightforward to check that $(\pi, f) \leexcomp (\pi',g)$ iff $F\WeiDoctrineOrder G$ (see \thref{def:Weihrauch doctrine}).
    
    This shows that the embedding preserves and reflects the partial order. Finally, it is easy to see that the map $(\pi_X,f) \mapsto F$ is surjective, since any function $H\function{X}{\powerset^* (\pca{A})^V}$ can be obtained via our embedding as the image of the pair $(\pr_X,h)$, where $h\function{X\times V}{\powerset^* (\pca{A})}$ is defined as $h (x,v):=H(x,v)$. Therefore, since the map is a surjective homomorphism of posets which also reflects the order, we can conclude that it is an isomorphism.
\end{proof} 

\subsection{Strong Weihrauch Doctrines}
\label{sec:strong_wei_doctrine}
We now show that we can adapt the techniques of the previous section to provide a categorical presentation of \emph{strong Weihrauch reducibility}. 

\begin{definition}[Strong Weihrauch reducibilty]
    \thlabel{def:wei_strong}
  Let $\pca{A}$ be a PCA, $\pca{A}'$ be an elementary sub-PCA of $\pca{A}$ and let $f, g\pmfunction{\pca{A}}{\pca{A}}$ be partial multi-valued functions on $\pca{A}$. We say that $f$ is \textbf{strongly Weihrauch reducible} to $g$, and write $f\strongweireducible g$, if there are $\ovln{h},\ovln{k}\in\subpca{A}$ such that
\[ (\forall G \vdash g)(\ovln{h} G \ovln{k} \vdash f).   \]

  \end{definition}
As anticipated, the intuition behind the strong Weihrauch reducibility is that a Weihrauch reduction $f\weireducible g$ is strong if the backward functional $\ovln{h}$ does not have access to the original input for $f$, but only to the solution of $g$. 

Analogously to what already observed in the context of Weihrauch reducibility, if $\ovln{h}$ and $\ovln{k}$ witness the reduction $f\strongweireducible g$ then $\ovln{h}$ witnesses the reduction $\dom(g)\medvedevreducible \dom(f)$. However, unlike what happens in the case of Weihrauch reducibility, the backward functional $\ovln{k}$ now witnesses a uniform Medvedev reduction between $f(p)$ and $g(\ovln{k}\cdot p)$.

We observe that the results of the previous subsection can be adapted to show that the \textbf{strong Weihrauch degrees} are isomorphic to the pure existential completion of another doctrine.
Again, it is convenient to rewrite the ordinary definition of strong Weihrauch reducibility as follows: 

\begin{definition}[Strong Weihrauch reducibility]
\thlabel{def:st_wei_relation}
Let $X,Z$ be two objects of $\PcaCat(\pca{A},\subpca{A})$. Given two functions $f\function{X}{\powerset^*(\pca{A})}$ and $g\function{Z}{\powerset^*(\pca{A})}$ then $f\strongweireducible g$ if and only if there is a computable function $\ovln{k}\function{X}{Z}$ and an element $\ovln{h}\in\subpca{A}$ such that
\[ (\forall p\in X)(\forall q\in g(\ovln{k}\cdot p))(\ovln{h}\cdot q\in f(p)) . \] 
\end{definition}

Now we adapt the definition of the (elementary) Weihrauch doctrine requiring that the map $\ovln{h}$ does not have access to the original input $p$.

\begin{definition}
    Given a PCA $\pca{A}$ with elementary sub-PCA $\subpca{A}$, we define the \textbf{elementary strong Weihrauch doctrine} $\doctrine{\PcaCat(\pca{A},\subpca{A})}{\strWeiElementary}$ as follows. 
    For every $X\subseteq \pca{A}$ and every pair of functions $f,g \in \powerset^*(\pca{A})^X$, we define $f \strWeiOrder g$ iff there is $\ovln{h}\in \subpca{A}$ such that 
    \[ (\forall p\in X)(\forall q\in g(p))(\ovln{h}\cdot q\in f(p)). \]

    This preorder induces an equivalence relation on functions in $\powerset^*(\pca{A})^X$. The doctrine $\strWeiElementary(X)$ is defined as the quotient of $\powerset^*(\pca{A})^X$ by the equivalence relation generated by $\strWeiOrder$. The partial order on the equivalence classes $[f]$ is the one induced by $\strWeiOrder$. 
\end{definition}

\begin{definition}[Strong Weihrauch doctrine]\thlabel{def: Strong Weihrauch doctrine}
    Given a PCA $\pca{A}$ with elementary sub-PCA $\subpca{A}$, the \textbf{strong Weihrauch doctrine} is the functor $\doctrine{\PcaCat(\pca{A},\subpca{A})}{\newstrWeiDoctrine}$ that maps $X\subseteq \pca{A}$ to the preorder $\newstrWeiDoctrine(X)$ defined as follows:
    \begin{itemize}
        \item objects are generalized Weihrauch predicates on $X$;
        \item the partial order is given by the poset reflection of the preorder defined as follows: let $Y,Z\subseteq \pca{A}$. We say that  $F\newstrWeiOrder G$, where $F\function{X}{\powerset^* (\pca{A})^Y}$ and $G\function{X}{\powerset^* (\pca{A})^Z}$, if there exists a computable function $ \ovln{k}\function{X \times Y}{Z} $ and $\ovln{h}\in \subpca{A}$ such that
        \[ (\forall \pairing{x,y} \in X\times Y)(\forall q\in G(x, \ovln{k}\cdot \pairing{x,y}))(\bar{h}\cdot q\in F(x,y)). \] 
        
        \end{itemize}
\end{definition}

In light of \thref{thm:weihrauch_lattice} and \thref{thm:Weihrauch isomorphism}, the following results are straightforward:

\begin{theorem}
    The strong Weihrauch lattice is isomorphic to $\newstrWeiDoctrine(\ovln{1})$. 
\end{theorem}

In particular, the (classical) strong Weihrauch degrees correspond to $\newstrWeiDoctrine(\ovln{1})$, when working with Kleene's second model.

\begin{theorem}[Strong Weihrauch isomorphism]\thlabel{thm:strong Weihrauch isomorphism}
Let $\pca{A}$ be a PCA and let $\pca{A}'$ be an elementary sub-PCA of $\pca{A}$. The strong Weihrauch doctrine 
is isomorphic to the pure existential completion
of the elementary strong Weihrauch doctrine. In symbols:
\[\newstrWeiDoctrine\equiv \compex{\strWeiElementary}.\]
\end{theorem}

\section{Generalizations of Weihrauch reducibility}
\label{sec:gen_weihrauch}

Weihrauch reducibility is often introduced in the more general context of partial multi-valued functions on represented spaces. We now briefly recall a few definitions in the context of computable analysis. For a more thorough presentation, the reader is referred to \cite{BGP17,Weihrauch00}.

\begin{definition}[represented space]\thlabel{def:represented space}
A \textbf{represented space} is a pair $(X, \delta_X)$, where $X$ is a set and $\delta_X\pfunction{\Baire}{X}$ is a surjective (partial) map called a \textbf{representation map}.
\end{definition}

For any given $x$ in $X$, the set $\delta_X^{-1}(x)$ is the set of $\delta_X$-\textbf{names} or $\delta_X$-\textbf{codes} for $x$. We avoid mentioning explicitly the representation map whenever there is no ambiguity.

\begin{definition}[realizer]\thlabel{def:rep_spaces_realizer}
Let $f\pmfunction{X}{Y}$ be a multi-valued function between the represented spaces $(X,\repmap{X})$ and $(Y,\repmap{Y})$. A \textbf{realizer} $F$ for $f$ (we write $F \vdash f$) is a function $F\pfunction{\Baire}{\Baire}$ such that  for all $p$ in the domain of $(f\circ \repmap{X})$ we have that $ (\repmap{Y}(F(p))\in f(\repmap{X}(p)))$. 
\end{definition}

Realizers are useful as they help us transfer properties of functions on the Baire space (such as computability or continuity) to multi-valued functions on the represented spaces. \thref{def:wei_classic} can be immediately extended to the case of multi-valued functions on represented spaces. We rewrite the definition here for the sake of readability.

\begin{definition}[Weihrauch reducibility for represented spaces]\thlabel{def:wei_classic_rep_spaces}
 Let $X,Y,Z,W$ be represented spaces and let $f\pmfunction{X}{Y}$ $g\pmfunction{Z}{W}$ be partial multi-valued functions. We say that $f$ is \textbf{Weihrauch reducible} to $g$, and write $f\weireducible g$, if there are two computable functionals $\ovln{h}, \ovln{k}\in \subpca{A}$ such that 
 \[ (\forall G\vdash g)( \ovln{h}(\id, G\ovln{k})\vdash f ), \] 
 where $\ovln{h}(\id, G\ovln{k}) := p \mapsto \ovln{h} \cdot \pairing{p, G(\ovln{k}\cdot p)}$ and $G \vdash g$ means that $G$ is a realizer of $g$ (as defined in \thref{def:rep_spaces_realizer}).
\end{definition}

Observe that the notion of Weihrauch reducibility for represented spaces depends on the representation maps. Indeed, expanding the previous definition, we obtain that $f\weireducible g$ iff if there are two computable functionals $\ovln{h}, \ovln{k}\in \subpca{A}$ s.t.\ 
\[ (\forall p\in\dom(f \circ \delta_X))( \ovln{k}\cdot p \in\dom(g\circ \delta_Z) \text{ and } (\forall q \in g\circ\delta_Z(\ovln{k}\cdot p))( \ovln{h}\cdot \pairing{p,q} \in f\circ \delta_X(p) )). \] 

While it is well-known that every Weihrauch degree has a representative that is a partial multi-valued function on $\pca{A}$, it is quite natural to ask whether the results of the previous sections could have been presented for doctrines whose base category is one of represented spaces, i.e.\  modest sets, or more generally, assemblies. 

The main purpose of this section is to show how our approach can be adapted to these settings. As we will explain in detail later, there are however some non-trivial technical obstacles in developing the theory using the ordinary categories of assemblies and modest sets. In particular, such obstacles regard the notion of morphism of assemblies which, as it is usually defined, is not able to properly support the notion of Weihrauch reducibility in the context of doctrines.

\subsection{Assemblies and Modest sets}\label{sec:rep_spaces}
We briefly recall some useful definitions regarding the notion of \emph{assembly} and \emph{modest sets} (in relative realizability), explaining how these can be seen as a categorification (and a generalization) of the notion of represented space. We refer to \cite{van_Oosten_realizability} for a complete presentation of these notions, and to \cite{Bauer2021,Frey2014AFS,HYLAND1988,BIRKEDAL2002115} for more specific applications. In the following definition, we mainly follow the notation used in \cite{Bauer2021}, which is closer to the usual used in computability.

\begin{definition}[Assemblies and Modest sets]\thlabel{def:assemblies}
Let $\pca{A}$ be a PCA. An \textbf{assembly} is a pair $X=(\barAsm{X},\relAsm{X})$ where $\barAsm{X}$ is a set and $\relAsm{X}\subseteq \pca{A}\times \barAsm{X}$ is a total relation, i.e.\ $(\forall x\in \barAsm{X})(\exists r\in \pca{A})( r\relAsm{X}x)$. An assembly $X$ is a \textbf{modest set} if  $r\relAsm{X}y$ and $ r\relAsm{X}x$ implies $x=y$, i.e.\ no two elements share a name. An assembly $X$  is \textbf{partitioned} if  $r\relAsm{X}x$ and  $r'\relAsm{X}x$ implies $r=r'$, i.e.\ every element has exactly one name.

\end{definition}
Given an assembly $X$, we will denote by $\support{X}$ the \emph{support} of $X$, namely the set $\support{X}:=\{a\in \pca{A}\st (\exists x\in \barAsm{X})( a\relAsm{X} x)\}$.

Clearly, every represented space  $(X,\repmap{X})$ (\thref{def:represented space}) gives rise to a unique modest set as in \thref{def:assemblies}, and vice-versa. In other words, the names ``modest set'' and ``represented space'' can be used interchangeably. The notion of realizer for represented spaces (\thref{def:rep_spaces_realizer}) can be straightforwardly extended to the notion of realizer for assemblies as follows: 

\begin{definition}[realizer for assemblies]\thlabel{def:asm_realizer}
Let $X$ and $Y$ be two assemblies, and let  $f\pmfunction{\barAsm{X}}{\barAsm{Y}}$ be a multi-valued function between their underlying sets. A \textbf{realizer} $F$ for $f$ (we write $F \vdash f$) is a (partial) function $F\pfunction{\pca{A}}{\pca{A}}$ such that  for all $r\relAsm{X}x$ implies $r\in \dom (F)$ and  $F(r)\relAsm{Y} f(x)$. 
\end{definition}

\begin{remark}\thlabel{rem:multi_func_modest_has_real}
    As mentioned, if $X$ and $Y$ are modest sets then every multi-valued function  $f\pmfunction{\barAsm{X}}{\barAsm{Y}}$ has a realizer (assuming the axiom of choice). 
    However, not every multi-valued function between assemblies admits a realizer.
\end{remark}
\begin{definition}[Morphism of assemblies]
Let $\pca{A}$ be a PCA and $\subpca{A}\subseteq \pca{A}$ be an elementary sub-PCA of $\pca{A}$.
A \textbf{morphism of assemblies} $f\function{X}{Y}$ is a function $f\function{\barAsm{X}}{\barAsm{Y}}$ which has a realizer in $\pca{A}'$, i.e.\ it is a function such that there exists an element $a\in \pca{A}'$ with the property that for every $r\relAsm{X}x$ then   $a\cdot r\downarrow$ and $a\cdot r\relAsm{X} f(x)$.
\end{definition}

\begin{definition}
Let $\pca{A}$ be a PCA and $\subpca{A}$ be an elementary sub-PCA of $\pca{A}$. We define the category of \textbf{assemblies} $\Assemblies(\pca{A},\subpca{A})$ as the subcategory of $ \Assemblies(\pca{A})$ having assemblies as objects and where morphisms are morphisms of assemblies. The category of \textbf{modest sets} is defined as the full subcategory $\Modest(\pca{A},\subpca{A})$ of $\Assemblies(\pca{A},\subpca{A})$ whose objects are modest sets. The category of \textbf{partitioned assemblies} $\partAsm(\pca{A},\subpca{A})$ is the full subcategory of $\Assemblies(\pca{A},\subpca{A})$ whose objects are partitioned assemblies.
\end{definition}

\begin{remark}\thlabel{rem:partition_modests_is_PcaCat}
    Observe that the full subcategory of $\Modest(\pca{A},\subpca{A})$ whose objects are partitioned, modest sets, corresponds exactly to the category $\PcaCat(\pca{A},\subpca{A})$ defined in \thref{def:PCA_cat}.
\end{remark}

\subsection{Realizer-based Weihrauch reducibility}
The notion of Weihrauch reducibility for assemblies has been introduced only very recently in \cite{Kihara2022rethinking, SchroederCCA2022}. While the definition is formally the same as \thref{def:wei_classic_rep_spaces}, there are some important differences with the classical notion of Weihrauch reducibility for represented spaces.

\begin{definition}[Realizer-based Weihrauch reducibility \cite{SchroederCCA2022}]
    \thlabel{def:rwei_relation}
    Let $\pca{A}$ be a PCA and $\subpca{A}\subseteq \pca{A}$ be an elementary sub-PCA of $\pca{A}$.
    Let $X,Y,Z,W$ be assemblies and let $f\pmfunction{\barAsm{X}}{\barAsm{Y}}$ and $g\pmfunction{\barAsm{Z}}{\barAsm{W}}$ be partial multi-valued functions. We say that $f$ is \textbf{realizer-based Weihrauch reducible} to $g$, and write $f \rweireducibile g$, if there are $\ovln{h}, \ovln{k}\in \subpca{A}$ such that 
    \[ (\forall G\vdash g)( \ovln{h}(\id, G\ovln{k})\vdash f ), \] 
    where $\ovln{h}(\id, G\ovln{k}) := p \mapsto \ovln{h} \cdot \pairing{p, G(\ovln{k}\cdot p)}$ and where $G \vdash g$ means that $G$ is a realizer of $g$ (\thref{def:asm_realizer}).

\end{definition}
\begin{remark}\thlabel{rem:wei_on_modest_is_part_case}
Notice that this coincides trivially with the standard definition of Weihrauch reducibility for represented spaces (see \thref{def:wei_classic_rep_spaces}) whenever $X,Y,Z,W$ are modest sets. However, while being a modest set is a special case of being an assembly, the notion of realizer for multi-valued functions on modest sets is ``weaker'' than the corresponding notion for assemblies. In fact, while every multi-valued function on modest sets has a realizer, the same is not true for multi-valued functions on assemblies, as mentioned in \thref{rem:multi_func_modest_has_real}. 
\end{remark}

The definition of realizer-based Weihrauch reducibility corresponds precisely to the generalization of Weihrauch reducibility for multi-represented spaces introduced in \cite[Def.\ 5.1]{Kihara2022rethinking}. For the sake of readability, this definition can be restated as follows: $f\rweireducibile g$ iff there are  $\ovln{h}, \ovln{k}\in \subpca{A}$ such that:

\begin{itemize}
    \item for every $x\in \dom(f)$ and $p\relAsm{X}x$, $\ovln{k}\cdot p\downarrow$ and there is $z\in \dom(g)$ such that $\ovln{k}\cdot p \relAsm{X} z$ and 
    \item for every $w\in g(z)$ and $q\relAsm{W}w$ there exists $\ovln{h}\cdot\pairing{p,q} \relAsm{Y} y$ such that $y\in f(x)$.
\end{itemize}

The variable $z$ in the second line is the same variable existentially quantified in the first line.

The previous description of realizer-based Weihrauch reducibility highlights a crucial aspect, namely that the forward functional $\ovln{k}$ witnessing the reduction $f\rweireducibile g$, in general, is not the realizer of a function between (the underlying sets of) the assemblies, i.e.\ it is not part of a morphism of assemblies. This is because different names for the same $x\in\dom(f)$ could be mapped to names for \emph{different} $z\in\dom(g)$ while, by definition, a morphism of assemblies has to send names of a given $x$ into names of another element $z$. This simple, but crucial fact, provides a serious technical obstacle in trying to adopt the ordinary categories of assemblies and modest sets as base categories for defining new doctrines abstracting realizer-based Weihrauch reducibility.

This observation would suggest that, to present a doctrine abstracting the notion of Weihrauch reducibility for represented spaces and assemblies, a natural choice could be that of considering a category whose objects are represented spaces and whose morphisms are multi-valued functions. However, this approach does not align well with the categorical framework of doctrines, primarily because represented spaces and multi-valued functions lack the structure of cartesian products.

Indeed, the category \textit{Mult} of multi-valued functions between sets has been studied in \cite{Pauly17Multi}. In \cite[Sec.\ 3.1]{Pauly17Multi}, Pauly observes that the category \textit{Mult} has a categorical product, but this product does not preserve computability, and so it does not work well for our purposes.

In summary, the challenge in extending our approach to assemblies and modest sets lies in defining a new notion of  ``morphism of assemblies'', so as to obtain a new category with cartesian products whose morphisms allow for a categorical presentation of realizer-based Weihrauch reducibility.

To this end, we introduce the following category $\extPcaCat(\pca{A},\subpca{A})$  where, as in the case of Weihrauch reducibility, the notion of morphism of the category is inspired by the properties of the forward functional. It will be then observed in \thref{rem_exAsm_equiv_parAsm} that such a change in the notion of morphisms between assemblies causes the new category to collapse into the category of partition assemblies.
Using $\extPcaCat(\pca{A},\subpca{A})$  will allow us to give a unified categorical presentation of $\rweireducibile$ (and later, also to extended Weihrauch reducibility $\extweireducible$). While the category of partitioned assemblies is well-studied in the literature, for the purposes of this paper, it is more convenient to work with the category $\extPcaCat(\pca{A},\subpca{A})$. This choice makes the connection between the realizer-based/extended Weihrauch reducibility and the corresponding doctrines apparent, and it renders some of the following proofs straightforward.

\begin{definition}\thlabel{def:cat_ext_pca}
    Let $\pca{A}$ be a PCA, and let $\pca{A}'$ be an elementary sub-PCA of $\pca{A}$. We denote by $\extPcaCat(\pca{A},\subpca{A})$ the category defined as follows:
    \begin{itemize}
        \item objects are assemblies  $X=(\barAsm{X}, \name_X)$;
        \item a morphism between $X$ and $Y$ is a pair $(a,\varphi)$, where $a\in \PcaCat(\pca{A},\subpca{A})(\support{X}, \support{Y})$ and $\varphi\function{\name_X}{Y}$ is a function such that, for every $p\name_X x$, $a\cdot p \name_Y \varphi(p,x)$.
    \end{itemize}
\end{definition}
\begin{remark}\thlabel{rem:morph_ex_induce_morph_relation}
    Observe that morphism  $(a,\varphi)\function{X}{Y}$ of $\extPcaCat(\pca{A},\subpca{A})$  induces a function $\name_X\to \name_Y$ assigning $(p,x)\mapsto (a\cdot p,\varphi (p,x))$. 
\end{remark}
In other words, the objects of $\extPcaCat(\pca{A},\subpca{A})$ are the same as of $\Assemblies(\pca{A},\subpca{A})$, the difference is in the notion of morphisms. 

It is straightforward to check that \thref{def:cat_ext_pca} provides a well-defined category, where the identity morphism $X\to X$ is given by the pair $(k,\pr_X)$, where $\pr_X\function{\relAsm{X}}{X}$ is the function mapping $(a,x)\mapsto x$, and the composition of $(a_1,\varphi_1)\function{X}{Y}$ and $(a_2,\varphi_2)\function{Y}{Z}$  is given by the morphism $(a_3,\varphi_3)\function{X}{Z}$ where $a_3$ is given by the composition of $a_1$ and $a_2$ in the category $\PcaCat(\pca{A},\pca{A}')$, while $\psi_3\function{\name_X}{Z}$ is given by the assignment $(p,x)\mapsto\varphi_2(a_1\cdot p,\varphi_1(p,x))$ (this is well-defined since, by definition of morphism, for every $p\name_X x$, $a_1\cdot p \name_Y \varphi_1(p,x)$).
\begin{remark}\thlabel{rem_exAsm_equiv_parAsm}
As already mentioned, the category $\extPcaCat(\pca{A},\subpca{A})$ is equivalent to the category $\partAsm(\pca{A},\subpca{A})$ of partitioned assemblies. Indeed, one can define a fully, faithful and essentially surjective functor $\extPcaCat(\pca{A},\subpca{A})\to \partAsm(\pca{A},\subpca{A})$ as follows: we can map an assembly $(X,\name_X)$ to the assembly $(\bar X, \name_{\bar X})$ where $\bar X:=\name_X$ and $\name_{\bar X}:=\{(p, (p,x)) \st p \name_X x\}$. It is immediate from the definition that $(\bar X, \name_{\bar X})$ is a partitioned assembly. Moreover, any morphism $(a,\varphi)$ of $\extPcaCat(\pca{A},\subpca{A})$ between the assemblies $(X,\name_X)$ and $(Y,\name_Y)$ is mapped to the morphism $f$ of (partitioned) assemblies between $(\bar X, \name_{\bar X})$ and $(\bar Y, \name_{\bar Y})$ defined as $f(p,x):=(a\cdot p, \varphi(p,x))$. It is direct to check that these assignments define an equivalence between the two categories. 
\end{remark}

The intuition behind the definition of $\extPcaCat(\pca{A},\subpca{A})$ is that a computation with a multi-valued function $f\pmfunction{X}{Z}$ between assemblies requires a (non-computable) step, as $p\in \support{X}$ may correspond to more than one $x\in\dom(f)$. In other words, on top of the computable maps $\ovln{k}$ and $\ovln{h}$, the reduction $f\rweireducibile g$ involves a (possibly non-computable) choice function $\varphi$ mapping $(p,x)$ with $p\name_X x \in \dom(f)$ to some element in $\dom(g)$. 

It is sometimes easier to think of a map between assemblies as a function having a ``public'' input (the name $p$ of some $x\in\dom(f)$) and a ``private'' input (the element $x$ of the assembly), where only the ``public'' input can be used in the computation (see also \cite[Observation 5.4 and the discussion thereafter]{Kihara2022rethinking}). 

Now, in order to show that the realizer-based Weihrauch reducibility can presented in the language of doctrines, we first rephrase  \thref{def:rwei_relation} in terms of reducibility of functions $F\function{\name_X}{\powerset^*(\pca{A})}$ and $G\function{\name_Z}{\powerset^*(\pca{A})}$, for some assemblies $X$ and $Y$.
    
\begin{proposition}
    \thlabel{thm:rW_and_extPCA}
    For every partial multi-valued function $f\pmfunction{X}{Y}$ between assemblies, let us define the map $F\function{\name_X}{\powerset^*(\pca{A})}$ as follows: for every $p\name_X x$, 
    \[ F(p,x) := \bigcup_{y\in f(x)} \{ q \in \pca{A} \st q \name_Y y \}. \] 
    Then $f\rweiequiv F$. Moreover, let $g\pmfunction{Z}{W}$ be a partial multi-valued function  between assemblies and let $G\function{\name_Z}{\powerset^*(\pca{A})}$ be its corresponding map. Then $f\rweireducibile g$ iff there are a morphism $(\ovln{k},\varphi)\in \extPcaCat(\pca{A},\pca{A}')(X,Z)$ and $\ovln{h}\in\subpca{A}$ such that 
    \[ (\forall p\name_X x)(\forall q \in G(\ovln{k}\cdot p, \varphi(p,x)))(\ovln{h}\cdot \pairing{p,q}\in F(p,x)).\tag{$\star$} \]
\end{proposition}
\begin{proof}
    The equivalence $f \rweiequiv F$ is straightforward from the definition. Let us show the equivalence between $f\rweireducibile g$ and $(\star)$. 

    For the right-to-left direction, it is immediate to see that $\ovln{k}$ and $\ovln{h}$ witness $f\rweireducibile g$. For the converse direction, assume $f\rweireducibile g$ via $\ovln{k},\ovln{h}$. For every $x\in \dom(f)$ and $p\name_X x$, let us define $\varphi(p,x):=z$, where $z\in \dom(g)$ is an element of $Z$ named by $\ovln{k}\cdot p$ witnessing the reduction $f\rweireducibile g$. In particular, the pair $(\ovln{k},\varphi)$ is a morphism of $\extPcaCat$ between $X$ and $Z$. To show that $(\ovln{k},\varphi)$ and $\ovln{h}$ prove the claim, fix $p\name_X x$. By definition, $\ovln{k}\cdot p \name_Z \varphi(p, x)$, and hence $G(\ovln{k}\cdot p, \varphi(p, x))=\bigcup_{y\in f(x)} \{ q \in \pca{A} \st q \name_Y y \}$. In particular, any $q\in G(\ovln{k}\cdot p, \varphi(p, x))$ is a name for some $t\in g(\varphi(p, x))$. This implies that $\ovln{h}\cdot \pairing{p,q} \name_Y y \in f( x)$, i.e.\ $\ovln{h}\cdot \pairing{p,q}\in F(p, x)$.
\end{proof}

The previous proposition establishes an equivalent formulation of realizer-based Weih\-rauch reducibility, which is better suited to presentation in the language of doctrines.  

\begin{definition}[elementary realizer-based Weihrauch doctrine]
    We define the \textbf{elementary realizer-based Weihrauch doctrine} $\doctrine{\extPcaCat(\pca{A},\subpca{A})}{\RWeiElementary}$ as follows: for every object $X$ of $\extPcaCat(\pca{A},\subpca{A})$, the objects $\RWeiElementary(X)$ are functions $f\function{\name_X}{\powerset^*(\pca{A})}$. For every pair of maps $f,g$ in $\powerset^*(\pca{A})^{\name_X}$, we define $f \RWeiOrderElementary g$ iff there is $\ovln{h}\in \subpca{A}$ such that 
    \[ (\forall p\name_X x)(\forall q \in g(p,x))(\ovln{h}\cdot \pairing{p,q}\in f(p,x)).\]
    As usual, the doctrine $\RWeiElementary(X)$ is the quotient of $\powerset^*(\pca{A})^{\name_X}$ by $\RWeiOrderElementary$, and the action of $\RWeiElementary$ on a morphism $(a,\varphi)\function{X}{Y}$ of $\extPcaCat(\pca{A},\subpca{A})$ is defined by the pre-composition with the function $\name_X\to\name_Y$ induced by $(a,\varphi)$ (see \thref{rem:morph_ex_induce_morph_relation}).
\end{definition}
Following the same idea of \thref{prop:structures of elementary weihrauch doctrine}, we can check that the elementary realizer-based Weihrauch doctrine is pure universal:
\begin{proposition}
    The elementary realizer-based Weihrauch doctrine $\RWeiElementary \colon \mathsf{extAsm}(\pca{A},$\linebreak[4]$\subpca{A})^{\op} \xymatrix@1{\ar[r] & \pos }$ is a pure universal doctrine: for every projection $\pi_Z\function{X\times Z}{Z}$, the morphism $\forall_{\pi_Z}\function{\RWeiElementary(X\times Z)}{\RWeiElementary(Z)}$ sending an element $f\in \RWeiElementary(X\times Z)$ to the element $\forall_{\pi_Z} (f)\in \RWeiElementary(Z)$ defined as
    \[\forall_{\pi_Z}  (f)(s,z):= \bigcup_{x\in X} \left\{ \pairing{p,q} \st p \name_X x \text{ and } q \in f((p,x),(s,z))  \right\}\]
    is right-adjoint to $\RWeiElementary_{\pi_Z}(g)$, i.e.\ for every $f\in \RWeiElementary(X\times Z)$ and $g \in \RWeiElementary(Z)$,
    \[g \RWeiOrderElementary \forall_{\pi_Z}(f)\iff \RWeiElementary_{\pi_Z} (g)\RWeiOrderElementary f.\]  
\end{proposition}
 
\begin{proof}
    Assume first that $g \RWeiOrderElementary \forall_{\pi_Z}(f)$ via $\ovln{a}$. Let $\ovln{b}\in\subpca{A}$ be defined as $\ovln{b}\cdot \pairing{\pairing{p,s},q} := \ovln{a}\cdot\pairing{s,\pairing{p,q}}$. We want to show that for every $\pairing{p,s}\name_{X\times Z} (x,z)$, 
    \[ (\forall q \in f((p,x),(s,z)))(\ovln{b}\cdot \pairing{\pairing{p,s},q} \in g\circ \pi_Z((p,x),(s,z)) = g(s,z)).\] 
    To this end, fix $\pairing{p,s}\name_{X\times Z} (x,z)$. Since $f$ is total on $X\times Z$, we only need to show that for every $q\in f((p,x),(s,z))$, $\ovln{b}\cdot \pairing{\pairing{p,s},q} \in g(s,z)$. This follows immediately from the definition of $\ovln{b}$, as $g \WeiOrderElementary \forall_{\pi_Z}(f)$ and $\pairing{p,q}\in \forall_{\pi_Z}(f)(s,z)$.

    For the right-to-left direction, we proceed analogously: assume $\RWeiElementary_{\pi_Z} (g)\RWeiOrderElementary f$ via $\ovln{c}$. Define $\ovln{d}\cdot \pairing{s,\pairing{p,q}}:=\ovln{c}\cdot\pairing{\pairing{p,s},q}$. To show that $g \RWeiOrderElementary \forall_{\pi_Z}(f)$, simply notice that for every $s\name_Z z$, and every $\pairing{p,q}\in \forall_{\pi_Z}(f)(s,z)$, the fact that $\RWeiElementary_{\pi_Z} (g)\RWeiOrderElementary f$ via $\ovln{c}$ implies that $\ovln{d}\cdot \pairing{s,\pairing{p,q}}=\ovln{c}\cdot\pairing{\pairing{p,s},q} \in \RWeiElementary_{\pi_Z} (g)((p,x),(s,z))= g(s,z)$.
\end{proof}

Now we introduce the realizer-based Weihrauch doctrine. For this, we introduce first the notion of \emph{generalized assembly-based Weihrauch predicate}.

\begin{definition}[generalized assembly-based Weihrauch predicate]
    A \textbf{generalized as\-sem\-bly-based Weihrauch predicate} on an assembly $X$ is a function
    \[F\function{\name_X}{\powerset^*(\pca{A})^{\name_Y}}\]
    for some assembly $Y$.
\end{definition}

\begin{definition}[Realizer-based Weihrauch doctrine]
    \thlabel{def:rweidoctrine}
    Given a PCA $\pca{A}$ with elementary sub-PCA $\subpca{A}$, the \textbf{realizer- based Weihrauch doctrine} is the functor $\RWeiElementary \colon \mathsf{extAsm}(\pca{A},$\linebreak[4]$\subpca{A})^{\op} \xymatrix@1{\ar[r] & \pos }$ that maps $(X,\name_X)$ to the preorder $\RWeiDoctrine(X)$ defined as follows:
    \begin{itemize}
        \item objects are generalized assembly-based Weihrauch predicates on $X$;
        \item the partial order is given by the poset reflection of the preorder defined as follows: let $(Y,\name_Y)$ and $(Z,\name_Z)$ be objects of $\extPcaCat(\pca{A},\subpca{A})$. For every $F\function{\name_X}{\powerset^*(\pca{A})^{\name_Y}}$ and $G\function{\name_X }{\powerset^*(\pca{A})^{\name_Z}}$, we say that $F\RWeiDoctrineOrder G$ if there exist a morphism $(\ovln{k},\varphi)\in \Hom_\extPcaCat( X\times Y, Z) $ and $\ovln{h}\in\subpca{A}$ such that 
        \[ (\forall \pairing{p,q}\name_{X\times Y} (x,y))( \forall t\!\in\!G((p,x), (\ovln{k}\cdot \pairing{p,q}, \varphi(\pairing{p,q},(x,y))) ))(\ovln{h}\cdot \pairing{\pairing{p,q},t}\!\in\! F((p,x),(q,y))).\]
    \end{itemize}
     As before, the action of $\RWeiDoctrine$ on a morphism $(a,\varphi)\function{X}{Y}$ of $\extPcaCat(\pca{A},\subpca{A})$ is defined by the pre-composition with the function $\name_X\to\name_Y$ induced by $(a,\varphi)$ (see \thref{rem:morph_ex_induce_morph_relation}).
\end{definition}
We now show that the previous doctrine provides the right categorification of realizer-based Weihrauch degrees.
\begin{theorem}[Realizer-based Weihrauch degrees]
    The realizer-based Weihrauch degrees are isomorphic to $\RWeiDoctrine(\mathbf{1})$, where $\mathbf{1}$ is the terminal object $(1,\name_1)$. 
\end{theorem}

\begin{proof}
    By definition of realizer-based Weihrauch doctrine, the objects of the poset $\RWeiDoctrine(\mathbf{1})$ can be identified with functions $\name_Y \to \powerset^*(\pca{A})$. Moreover, if $F\function{\name_Y}{\powerset^*(\pca{A})}$ and $G\function{\name_Z}{\powerset^*(\pca{A})}$ are in $\RWeiDoctrine(\mathbf{1})$ then $F\RWeiDoctrineOrder G$ iff there are a morphism $(\ovln{k},\varphi)$ between $Y$ and $Z$ and $\ovln{h}\in\subpca{A}$ such that
     \[ (\forall p \name_Y y)(\forall q\in G(\ovln{k}\cdot p, \varphi(p,y)))(\ovln{h}\cdot \pairing{p,q}\in F(p,y)), \]
    which corresponds to \thref{def:rwei_relation}.
\end{proof}

As before, when working with Kleene's second model, $\RWeiDoctrine(\mathbf{1})$ corresponds precisely to the realizer-based Weihrauch degrees.

We conclude our analysis of realizer-based Weihrauch reducibility by showing that the realizer-based Weihrauch doctrine is a pure existential completion.
\begin{theorem}
    The realizer-based Weihrauch doctrine is isomorphic to the pure existential completion of the elementary realizer-based Weihrauch doctrine. In symbols:
    \[\RWeiDoctrine\equiv \compex{\RWeiElementary}.\]
    In particular, for every PCA $\pca{A}$ with elementary sub-PCA $\subpca{A}$ and for every object $(X,\name_X)$ of $\extPcaCat(\pca{A},\subpca{A})$, the map 
    \[(\pi_X,f)\mapsto F \]
    where $F\function{\name_X }{\powerset^*(\pca{A})^{\name_Y}}$ is defined as $F(p,x):=f((p,x),\cdot)$, is an isomorphism of posets between $(\compex{\RWeiElementary}(X),\leq_{\exists})$ and $(\RWeiDoctrine(X),\RWeiDoctrineOrder)$.
\end{theorem}
\begin{proof}
    Recall that 
    \[ (\pi  ,f) \leexcomp (\pi',g) \iff (\exists (\ovln{k},\varphi)\in \extPcaCat(\pca{A},\subpca{A})(X\times Y,Z))(f \RWeiOrderElementary \RWeiElementary_{\angbr{\pi_X}{k}}(g)), \]
    where $\RWeiElementary_{\angbr{\pi_X}{k}}(g) \colon= ((p,x),(q,y))\mapsto g((p,x),(\ovln{k}\cdot\pairing{p,q},\varphi((p,x),(q,y)) )$.
    
    We first show that the assignment $(\pi_X,f)\mapsto F$ preserves and reverses the order. Let $(\pi ,f)$ and $(\pi',g)$ be two elements of $\compex{\RWeiElementary}(X)$. By expanding the definition of $\RWeiOrderElementary$, we obtain that $(\pi  ,f) \leexcomp (\pi',g) $ iff there is a morphism $(\ovln{k},\varphi)\in \extPcaCat(\pca{A},\subpca{A})(X\times Y,Z)$ and $\ovln{h}\in\subpca{A}$ such that 

    \[ (\forall \pairing{p,q}\name_{X\times Y} (x,y))(\forall t \in g((p,x),(\ovln{k}\cdot\pairing{p,q},\varphi(\pairing{p,q},(x,y))))) (\ovln{h}\cdot \pairing{\pairing{p,q},t} \in f((p,x),(q,y)))\tag{$\star$}\] 
    Let $F\function{\name_X }{\powerset^*(\pca{A})^{\name_Y}}$ and $G\function{\name_X }{\powerset^*(\pca{A})^{\name_Z}}$ be the images of $(\pi,f)$ and $(\pi',g)$ respectively. Observe that, writing $F((p,x),(q,y))=f((p,x),(q,y))$ (analogously for $G$) and substituting $F$ and $G$ in ($\star$), it is straightforward to check that $(\pi, f) \leexcomp (\pi',g)$ iff $F\RWeiDoctrineOrder G$ (see the definition).
    
    This shows that the embedding preserves and reflects the partial order. Finally, it is easy to see that the map $(\pi_X,f) \mapsto F$ is surjective, since any function $H\function{\name_X}{\powerset^* (\pca{A})^{\name_V}}$ can be obtained via our embedding as the image of the pair $(\pr_X,h)$, where $h\function{\name_X\times \name_V}{\powerset^* (\pca{A})}$ is defined as $h ((p,x),(q,v)):=H((p,x),(q,v))$. Therefore, since the map is a surjective homomorphism of posets which also reflects the order, we can conclude that it is an isomorphism.
\end{proof} 
\begin{remark}\thlabel{rem:real_bs_W_on_modest}
    As anticipated \thref{rem:wei_on_modest_is_part_case}, if consider the realizer-based Weihrauch doctrine restricted to the full subcategory of $\extPcaCat(\pca{A},\subpca{A})$ whose objects are modest sets, we obtain precisely a doctrine abstracting the notion of Weihrauch reducibility for represented spaces. 
\end{remark}

\subsection{Extended Weihrauch reducibility}
Recently Bauer \cite{Bauer2021} introduced another generalization of Weihrauch reducibility, called \emph{extended Weihrauch reducibility}, that can be seen as another way to generalize Weihrauch reducibility to multi-represented spaces. As anticipated, we show that the category $\extPcaCat(\pca{A},\subpca{A})$ can be used to describe the extended Weihrauch degrees in terms of doctrines. 

We start by recalling the main definitions from \cite{Bauer2021}:

\begin{definition}[Extended Weihrauch reducibility {\cite[Def.\ 3.7]{Bauer2021}}]
    \thlabel{def:extwei_relation}
    Let $\pca{A}$ be a  PCA. An \textbf{extended Weihrauch predicates} is a function $f\function{\pca{A}}{\powerset\powerset(\pca{A})}$.

    Given two extended Weihrauch predicates $f,g$, we say that$f$ is \textbf{extended-Weihrauch reducible} to $g$, and write $f \extweireducible g$ if there are $\ovln{k}, \ovln{h} \in \subpca{A}$ such that 
\begin{itemize}
    \item for every $p\in \pca{A}$ such that $f(p)\neq \emptyset$, $\ovln{k}\cdot p\downarrow$ and $g(\ovln{k}\cdot p)\neq \emptyset$;
    \item for every $A \in f(p)$ there is $B \in g(\ovln{k}\cdot p)$ such that for every $q \in B$, $\ovln{h}\cdot \pairing{p,q} \downarrow$ and $\ovln{h}\cdot \pairing{p,q} \in A$.
\end{itemize}
\end{definition}
Observe that, while extended Weihrauch reducibility applies to extended Weihrauch predicates, we can rewrite the above definition as a preorder on partial multi-valued functions on elements of $\extPcaCat(\pca{A},\subpca{A})$. This follows the same ideas used in the proof of \cite[Prop.\ 3.8]{Bauer2021}: given $f\function{\pca{A}}{\powerset\powerset(\pca{A})}$, let $X_f=(\barAsm{X_f}, \name_{X_f})$ be the object of $\extPcaCat(\pca{A},\subpca{A})$ defined as follows:
\begin{gather*}
    \barAsm{X_f} := \bigcup_{p\in\dom(f)} f(p) = \{ A \in \powerset(\pca{A}) \st (\exists p\in \dom(f))( A \in f(p))\};\\
    p \name_{X_f} A :\iff A \in f(p).
\end{gather*}
We then define $F\function{\name_{X_f}}{\powerset(\pca{A})}$ as $F(p,A):=A$.

Observe that, unlike what happens for $\rweireducibile$, $F$ need not be a total map. In other words, it could be that $F(p,A)=\emptyset$ for some $p\name_{X_f} A$. This is a crucial difference between $\rweireducibile$ and $\extweireducible$.
\begin{proposition}\thlabel{prop:equiv_rew_of_ext_wei}
Given two extended predicates $f,g$, $f \extweireducible g$ iff there are a morphism $(\ovln{k},\varphi)\in \extPcaCat(\pca{A},\subpca{A})(X_f, X_g)$ and $\ovln{h}\in\subpca{A}$ such that 
\[ (\forall p\name_{X_f} A)(\forall q \in G(\ovln{k}\cdot p, \varphi(p,A)))(\ovln{h}\cdot \pairing{p,q}\in F(p,A)). \]
\end{proposition}
This characterization is formally identical to $(\star)$ in \thref{thm:rW_and_extPCA}, and it can be proved using the same ideas. However, the fact that we do not require $F$ and $G$ to be total implies that $f\extweireducible g$ is vacuously true whenever $\emptyset\in g(q)$ for every $q\in\dom(g)$. In fact, the top extended Weihrauch degree has a representative $g$ with $\dom(g)=\{\ovln{a}\}$ for some $\ovln{a}\in\subpca{A}$ and $g(\ovln{a})=\{\emptyset\}$. In other words, realizer-based Weihrauch reducibility and extended Weihrauch reducibility agree if we restrict the latter to $\lnot\lnot$-\textbf{dense extended Weihrauch predicates}, namely extended Weihrauch predicates of the type $f\function{\pca{A}}{\powerset\powerset(\pca{A})}$ with $\emptyset \notin f(p)$ for every $p$. 

Analogously to what is obtained for the realizer-based Weihrauch reducibility, we can define an \textbf{extended Weihrauch doctrine} that is isomorphic to the pure existential completion of a pure universal doctrine and that, when evaluated on $\mathbf{1}$, is isomorphic to the extended Weihrauch degrees. The definition and results closely follow the ones obtained for the realizer-based Weihrauch reducibility.

\begin{definition}[elementary extended Weihrauch doctrine]
    We define the \textbf{elementary extended Weihrauch doctrine} $\doctrine{\extPcaCat(\pca{A},\subpca{A})}{\EWeiElementary}$ as follows: for every object $(X,\name_X)$ of $\extPcaCat(\pca{A},\subpca{A})$, the objects $\EWeiElementary(X)$ are functions $f\function{\name_X}{\powerset(\pca{A})}$. For every pair of maps $f,g$ in $\powerset(\pca{A})^{\name_X}$, we define $f \EWeiOrderElementary g$ iff there is $\ovln{h}\in \subpca{A}$ such that 
    \[ (\forall p\name_X x)(\forall q \in g(p,x))(\ovln{h}\cdot \pairing{p,q}\in f(p,x)).\]
    As usual, the doctrine $\EWeiElementary(X)$ is the quotient of $\powerset(\pca{A})^{\name_X}$ by $\EWeiOrderElementary$, and the action of $\EWeiElementary$ on the morphisms of $\extPcaCat(\pca{A},\subpca{A})$ is defined by the (suitable) pre-composition.
\end{definition}

\begin{proposition}
    The elementary extended Weihrauch doctrine $\EWeiElementary \colon \mathsf{extAsm}(\pca{A},$\linebreak[4]$\subpca{A})^{\op} \xymatrix@1{\ar[r] & \pos }$ is a pure universal doctrine: for every projection $\pi_Z\function{X\times Z}{Z}$, the morphism $\forall_{\pi_Z}\function{\EWeiElementary(X\times Z)}{\EWeiElementary(Z)}$ sending an element $f\in \EWeiElementary(X\times Z)$ to the element $\forall_{\pi_Z} (f)\in \EWeiElementary(Z)$ defined as
    \[\forall_{\pi_Z}  (f)(s,z):= \bigcup_{x\in X} \left\{ \pairing{p,q} \st p \name_X x \text{ and } q \in f((p,x),(s,z))  \right\}\]
    is right-adjoint to $\EWeiElementary_{\pi_Z}(g)$, i.e.\ for every $f\in \EWeiElementary(X\times Z)$ and $g \in \EWeiElementary(Z)$,
    \[g \EWeiOrderElementary \forall_{\pi_Z}(f)\iff \EWeiElementary_{\pi_Z} (g)\EWeiOrderElementary f.\]  
\end{proposition}

\begin{proof}
    Assume first that $g \EWeiOrderElementary \forall_{\pi_Z}(f)$ via $\ovln{a}$. Let $\ovln{b}\in\subpca{A}$ be defined as $\ovln{b}\cdot \pairing{\pairing{p,s},q} := \ovln{a}\cdot\pairing{s,\pairing{p,q}}$. We want to show that for every $\pairing{p,s}\name_{X\times Z} (x,z)$, 
    \[ (\forall q \in f((p,x),(s,z)))(\ovln{b}\cdot \pairing{\pairing{p,s},q} \in g\circ \pi_Z((p,x),(s,z)) = g(s,z)).\] 
    To this end, fix $\pairing{p,s}\name_{X\times Z} (x,z)$ and fix $q\in f((p,x),(s,z))$. Observe that $\pairing{p,q}\in \forall_{\pi_Z}(f)(s,z)$, therefore we immediately obtain
    \[ \ovln{b}\cdot \pairing{\pairing{p,s},q} = \ovln{a}\cdot\pairing{s,\pairing{p,q}} \in g(s,z). \]    

    For the right-to-left direction, we proceed analogously: assume $\EWeiElementary_{\pi_Z} (g)\EWeiOrderElementary f$ via $\ovln{c}$. Define $\ovln{d}\cdot \pairing{s,\pairing{p,q}}:=\ovln{c}\cdot\pairing{\pairing{p,s},q}$. To show that $g \EWeiOrderElementary \forall_{\pi_Z}(f)$, simply notice that for every $s\name_Z z$, and every $\pairing{p,q}\in \forall_{\pi_Z}(f)(s,z)$ (in particular, $q\in f((p,x),(s,z))$), the fact that $\EWeiElementary_{\pi_Z} (g)\EWeiOrderElementary f$ via $\ovln{c}$ implies that $\ovln{d}\cdot \pairing{s,\pairing{p,q}}=\ovln{c}\cdot\pairing{\pairing{p,s},q} \in \EWeiElementary_{\pi_Z} (g)((p,x),(s,z))= g(s,z)$.
\end{proof}

\begin{definition}[generalized extended Weihrauch predicate]
    A \textbf{generalized extended Weihrauch predicate} on an assembly $X$ is a function
    \[F\function{\name_X}{\powerset(\pca{A})^{\name_Y}}\]
    for some assembly $Y$.
\end{definition}

\begin{definition}[Extended Weihrauch doctrine]
    Given a PCA $\pca{A}$ with elementary sub-PCA $\subpca{A}$, the \textbf{extended Weihrauch doctrine} is the functor $\doctrine{\extPcaCat(\pca{A},\subpca{A})}{\EWeiDoctrine}$ that maps $(X,\name_X)$ to the preorder $\EWeiDoctrine(X)$ defined as follows:
    \begin{itemize}
        \item objects are generalized extended Weihrauch predicates;
        \item the partial order is given by the poset reflection of the preorder defined as follows: let $Y$ and $Z$ be objects of $\extPcaCat(\pca{A},\subpca{A})$. For every $F\function{\name_X}{\powerset(\pca{A})^{\name_Y}}$ and $G\function{\name_X }{\powerset(\pca{A})^{\name_Z}}$, we say that $F\EWeiDoctrineOrder G$ if there exist a morphism $(\ovln{k},\varphi)\in \extPcaCat(\pca{A},\subpca{A})( X\times Y, Z) $ and $\ovln{h}\in\subpca{A}$ such that 
        \[ (\forall \pairing{p,q}\name_{X\times Y} (x,y))( \forall t\!\in\!G((p,x), (\ovln{k}\cdot \pairing{p,q}, \varphi(\pairing{p,q},(x,y))) ))(\ovln{h}\cdot \pairing{\pairing{p,q},t}\!\in\!F((p,x),(q,y))).\]
    \end{itemize}
\end{definition}
\begin{theorem}[Extended Weihrauch degrees]
    The extended Weihrauch degrees are isomorphic to $\EWeiDoctrine(\mathbf{1})$.
\end{theorem}

\begin{proof}
    By definition of extended Weihrauch doctrine, the objects of the poset $\EWeiDoctrine(\mathbf{1})$ can be identified with functions $\name_Y \to \powerset(\pca{A})$. Moreover, if $F\function{\name_Y}{\powerset(\pca{A})}$ and $G\function{\name_Z}{\powerset(\pca{A})}$ are in $\EWeiDoctrine(\mathbf{1})$ then $F\EWeiDoctrineOrder G$ iff there are a morphism $(\ovln{k},\varphi)$ between $Y$ and $Z$ and $\ovln{h}\in\subpca{A}$ such that
     \[ (\forall q \name_Y y)(\forall t\in G(\ovln{k}\cdot q, \varphi(q,y)))(\ovln{h}\cdot \pairing{q,t}\in F(q,y)), \]
    which, by \thref{prop:equiv_rew_of_ext_wei}, corresponds to \thref{def:extwei_relation}.
\end{proof}

\begin{theorem}
    The extended Weihrauch doctrine is isomorphic to the pure existential completion of the elementary extended Weihrauch doctrine. In symbols:
    \[\EWeiDoctrine\equiv \compex{\EWeiElementary}.\]
    In particular, for every PCA $\pca{A}$ with elementary sub-PCA $\subpca{A}$ and for every object $X$ of $\extPcaCat(\pca{A},\subpca{A})$, the map 
    \[(\pi_X,f)\mapsto F \]
    where $F\function{\name_X }{\powerset(\pca{A})^{\name_Y}}$ is defined as $F(p,x):=f((p,x),\cdot)$, is an isomorphism of posets between $(\compex{\EWeiElementary}(X),\leq_{\exists})$ and $(\EWeiDoctrine(X),\EWeiDoctrineOrder)$.
\end{theorem}

\begin{proof}
    Recall that 
    \[ (\pi  ,f) \leexcomp (\pi',g) \iff (\exists (\ovln{k},\varphi)\in \extPcaCat(\pca{A},\subpca{A})(X\times Y,Z))(f \EWeiOrderElementary \EWeiElementary_{\angbr{\pi_X}{k}}(g)), \]
    where $\EWeiElementary_{\angbr{\pi_X}{k}}(g) \colon= ((p,x),(q,y))\mapsto g((p,x),(\ovln{k}\cdot\pairing{p,q},\varphi((p,x),(q,y)) )$.
    
    We first show that the assignment $(\pi_X,f)\mapsto F$ preserves and reverses the order. Let $(\pi ,f)$ and $(\pi',g)$ be two elements of $\compex{\EWeiElementary}(X)$. By expanding the definition of $\EWeiOrderElementary$, we obtain that $(\pi  ,f) \leexcomp (\pi',g) $ iff there is a morphism $(\ovln{k},\varphi)\in \extPcaCat(\pca{A},\subpca{A})(X\times Y,Z)$ and $\ovln{h}\in\subpca{A}$ such that 
    \[ (\forall \pairing{p,q}\name_{X\times Y} (x,y))(\forall t \in g((p,x),(\ovln{k}\cdot\pairing{p,q},\varphi(\pairing{p,q},(x,y))))) (\ovln{h}\cdot \pairing{\pairing{p,q},t} \in f((p,x),(q,y)))\tag{$\star$}\] 
    Let $F\function{\name_X }{\powerset(\pca{A})^{\name_Y}}$ and $G\function{\name_X }{\powerset(\pca{A})^{\name_Z}}$ be the images of $(\pi,f)$ and $(\pi',g)$ respectively. Observe that, writing $F((p,x),(q,y))=f((p,x),(q,y))$ (analogously for $G$) and substituting $F$ and $G$ in ($\star$), it is straightforward to check that $(\pi, f) \leexcomp (\pi',g)$ iff $F\EWeiDoctrineOrder G$ (see the definition).
    
    This shows that the embedding preserves and reflects the partial order. Finally, it is easy to see that the map $(\pi_X,f) \mapsto F$ is surjective, since any function $H\function{\name_X}{\powerset (\pca{A})^{\name_V}}$ can be obtained via our embedding as the image of the pair $(\pr_X,h)$, where $h\function{\name_X\times \name_V}{\powerset(\pca{A})}$ is defined as $h ((p,x),(q,v)):=H((p,x),(q,v))$. Therefore, since the map is a surjective homomorphism of posets which also reflects the order, we can conclude that it is an isomorphism.
\end{proof} 
As observed in \cite{Bauer2021}, the (classical) Weihrauch degrees are isomorphic to the restriction of the extended Weihrauch degrees to the $\lnot\lnot$-\textbf{dense modest extended Weihrauch predicates}, namely $\lnot\lnot$-dense predicates $f\function{\pca{A}}{\powerset\powerset(\pca{A})}$ such that, for every $p\in\dom(f)$, $|f(p)|=1$. These correspond precisely to objects of the elementary extended Weihrauch doctrine $\EWeiElementary(X)$ such that $(X,\name_X)$ is a modest set. In other words, the classical Weihrauch reduction for partial multi-valued functions on represented spaces can be again, as observed in \thref{rem:real_bs_W_on_modest}, obtained by considering the restriction of extended Weihrauch doctrine on the full-subcategory of $\extPcaCat(\pca{A},\subpca{A})$ whose objects are modest sets. 
\subsection{Some remarks}
We conclude this section by highlighting an interesting connection between Medvedev and extended Weihrauch reducibility. The definition of extended Weihrauch reducibility can be naturally strengthened (to obtain its \emph{strong} counterpart) by requiring that the map $\ovln{h}$ does not have access to the original input. 

\begin{definition}[Extended strong Weihrauch reducibility]
    If $f,g$ are extended Weihrauch predicates, we say that $f$ is \emph{extended strong Weihrauch reducible} to $g$, and write $f \extstrongweireducible g$ if there are $\ovln{k}, \ovln{h} \in \subpca{A}$ such that 
\begin{itemize}
    \item for every $p\in \pca{A}$ such that $f(p)\neq \emptyset$, $\ovln{k}\cdot p\downarrow$ and $g(\ovln{k}\cdot p)\neq \emptyset$;
    \item for every $A \in f(p)$ there is $B \in g(\ovln{k}\cdot p)$ such that for every $q \in B$, $\ovln{h}\cdot q \downarrow$ and $\ovln{h}\cdot q \in A$.
\end{itemize}
\end{definition}

Since the map $\ovln{h}$ does not have access to the original input, the definition of extended strong Weihrauch degrees could be given in the more general context of functions $X\to\powerset\powerset{(\pca{A})}$.

As observed above, strong Weihrauch reducibility and Medvedev reducibility are closely connected, as a strong Weihrauch reducibility $f \strongweireducible g$ can be seen as a Medvedev reducibility between the domains and a uniform Medvedev reducibility between the images. We now show that we can write $\extstrongweireducible$ in terms of the full existential completion of the Medvedev doctrine $\MedvedevDoctrine$. 

To this end, we first define a new doctrine:
\begin{definition}
   The doctrine $\doctrine{\set}{\D}$ is defined as follows: for every set $X$, the objects of $\D(X)$ are functions of type $F\function{R}{\powerset(\pca{A})}$, where $R\subseteq X\times \powerset(\pca{A})$ is a relation. As a notational convenience, let us write $x \name_R A$ if $(x,A)\in R$. The partial order is given by the poset reflection of the preorder defined as follows:  $F\le_D G$ for $G\function{S}{\powerset(\pca{A})}$, iff there are a function $\varphi\function{R}{\powerset(\pca{A})}$ and $\ovln{h}\in \subpca{A}$ such that, for every $x\name_R A$, $x\name_S \varphi(x,A)$ and $\ovln{h}\cdot G(x,\varphi(x,A)) \subseteq F(x,A)$, i.e.
\[ (\forall q \in G(x,\varphi(x,A)))(\ovln{h}\cdot q \in F(x,A)). \]
The action of the functor $\D$ on morphisms of $\set$ is defined, as usual, by pre-composition.
\end{definition}
\begin{remark}
Observe that, if we take $X\subseteq \pca{A}$, a relation on $X\times \powerset(\pca{A})$ can be identified with the assembly $\ran(R)$ whose set of names is exactly $\dom(R)$.

In particular, $(\ran(R),R)$ is an object of $\extPcaCat(\pca{A},\subpca{A})$. In this case, $F \le_D G$ can be equivalently restated by asking that there is a $\varphi$ and $\ovln{h}\in\subpca{A}$ such that $(\id, \varphi)\in \extPcaCat(\pca{A},\subpca{A})(\ran(R),\ran(S))$ and
\[ (\forall x\name_R A)(\forall q \in G(x,\varphi(x,A)))(\ovln{h}\cdot q \in F(x,A)). \] 
In other words, this corresponds to say that we have a reduction $F\extweireducible G$ witnessed by $(\id,\varphi)$ and $\ovln{h}$, and $\ovln{h}$ need not access the original input $x$. 
\end{remark}
We now prove that the full existential completion of $\MedvedevDoctrine$ is equivalent to $\D$. 
\begin{theorem}\thlabel{thm:D_is_full_ex_comp_Med}
$\compexfull{\MedvedevDoctrine}\cong \D$.

\end{theorem}
\begin{proof}
For every set $X$, we show that $(\compexfull{\MedvedevDoctrine}(X),\lefullexcomp)$ is isomorphic to $(\D(X),\le_D)$. We start by defining the mapping $(\compexfull{\MedvedevDoctrine}(X),\lefullexcomp)\mapsto (\D(X),\le_D)$: this maps an element $(f,\alpha)$ to the function $F_{(f,\alpha)}\function{R_{(f,\alpha)}}{\powerset(\pca{A})}$ where $R_{(f,\alpha)}\subseteq X\times \powerset(\pca{A})$ is the set 
\[R_{(f,\alpha)}:=\{(x,A)\in X\times \powerset (\pca{A})\st (\exists y\in f^{-1}(x))(A= \alpha(y))\}\]
 and $F_{(f,\alpha)}(x,A):=A$. Now we show that this function preserves the order.
    
    Fix $(f,\alpha),(g,\beta)\in \compexfull{\MedvedevDoctrine}(X)$, with $\alpha\in \powerset(\pca{A})^Y$ and $\beta\in \powerset(\pca{A})^Z$. 
    Let us first show that $(f,\alpha) \lefullexcomp (g,\beta)$ implies $F_{(f,\alpha) }\le_D F_{(g,\beta)}$. Fix $k\function{Y}{Z}$ and $\ovln{h}\in\pca{A}$ witnessing $(f,\alpha) \lefullexcomp (g,\beta)$, i.e.\ $f=g k$ and $\alpha \MedvedevOrder \beta \circ k$ via $\ovln{h}$. Let $c\function{R_{(f,\alpha)}}{Y}$ be a choice function that maps every pair $(x,A)$ with $x \name_{R_{(f,\alpha)}} A$, to some $y\in \alpha^{-1}(A)$. By definition, if $x\name_{R_{(f,\alpha)}} A$ then  $\alpha^{-1}(A)\neq \emptyset$, hence $c$ is well-defined. 
    We define $\varphi(x,A) := \beta\circ k\circ c(x,A)$. We claim that $\ovln{h}$ and $\varphi$ witness $F_{(f,\alpha)}\le_D F_{(g,\beta)}$. Observe first that, for every $x \name_{R_{(f,\alpha)}} A$, $x\name_{R_{(g,\beta)}} \varphi(x,A)$. This follows from the fact that, for every $y\in Y$, $g(k(y))=f(y)$ (by definition of full existential completion). In particular, $g(k(c(x,A)))=x$, and therefore $(x, \beta\circ k\circ c(x,A))=(x,\varphi(x,A))\in R_{(g,\beta)}$. 
    Moreover, for every $x\name_{R_{(g,\beta)}} \varphi(x,A)$, letting $y := c(x,A)$ we have 
    \[\ovln{h}\cdot F_{(g,\beta)}(x,\varphi(x,A))= \ovln{h}\cdot \varphi(x,A) = \ovln{h}\cdot (\beta \circ k(y)) \subseteq \alpha(y) = A=F_{(f,\alpha)}(x,A). \]

   Now we show that our assignment also reflects the order: if $F_{(f,\alpha) }\le_D F_{(g,\beta)}$ is witnessed by $\ovln{h},\varphi$, then we define $k\function{Y}{Z}$ as a choice function that maps every $y$ in $Y$ to some element in $\{ z\in Z \st \beta(z) =\varphi(f(y),\alpha(y)) \}$. Notice that $k$ is well-defined: indeed, by hypothesis, $f(y) \name_{R_{(g,\beta)}} \varphi(f(y),\alpha(y))$, i.e.\ there is some $z\in Z$ such that $g(z)=f(y)$ and $\beta(z)=\varphi(f(y),\alpha(y))$. This also shows that, for every $y\in Y$,  $f(y)=g(k(y))$.
    
    To prove that $\ovln{h}$ and $k$ witness $(f,\alpha) \lefullexcomp (g,\beta)$ it is enough to notice that, for every $y$ in $Y$, 
    \[ \ovln{h}\cdot \beta(k(y)) \ovln{h}\cdot F_{(g,\beta)}(f(y),\varphi(f(y),\beta(k(y)))) \subseteq F_{(f,\alpha)}(f(y),\alpha(y))= \alpha(y) .  \]
Therefore, we have proved that the assignment $(\compexfull{\MedvedevDoctrine}(X),\lefullexcomp)\mapsto (\D(X),\le_D)$ determines a morphism of posets which also reflects the order. To conclude that it is an isomorphism, it is enough to show that it is surjective. So, let us consider an element $F\function{R}{\powerset (\pca{A})}$ of $\D(X)$, and let us consider the element $(\pr_X\function{R}{X},F)$ of $\compexfull{\MedvedevDoctrine}(X)$. We claim that $F\le_D F_{(\pr_X,F)}$ and $ F_{(\pr_X,F)}\le_D F$, i.e.\ they are equivalent in $\D(X)$. First observe that, by definition, we have that
\[R_{(\pr_X,F)}:=\{(x,A)\in X\times \powerset (\pca{A})\st (\exists A'\in \powerset (\pca{A}))(x\name_R A' \mbox{ and } A =F(x,A')) \}.\]
We claim that $F\le_D F_{(\pr_X,F)}$ via $\varphi:=F$ and the identity of $\pca{A}'$. First, by definition of $R_{(\pr_X,F)}$, we have that  $x \name_R A$ implies that $x\name_{R_{(\pr_X,F)}}\varphi (x,A) =F(x,A)$. Moreover,  
\[F_{(\pr_X,F)}(x,\varphi(x,A))=\varphi (x,A)=F(x,A).\]

For the converse direction, let $\varphi\function{R_{(\pr_X,F)}}{\powerset (\pca{A})}$ defined as $\varphi (x,A):=A'$ for (a choice of) an $A'$ satisfying $x\name_R A'$ and $A=F(x,A')$. We claim that $\varphi$ and the identity of $\subpca{A}$ witness $F_{(\pr_X,F)}\le_D F$. Notice that, by definition of $R_{(\pr_X,F)}$, $\varphi$ is well-defined. By definition of $R_{(\pr_X,F)}$, if $x\name_{R_{(\pr_X,F)}}A$ then we have that $x\name_R \varphi (x,A)= A'$. Moreover,
\[F(x,\varphi (x,A))=F(x,A')=A=F_{(\pr_X,F)}(x,A)\]
by definition of $\varphi$ and $F_{(\pr_X,F)}$. Therefore, we can conclude that $F_{(\pr_X,F)}\le_D F$, and hence that $F_{(\pr_X,F)}=F$ in the poset $\D(X)$. This concludes the proof that $\D\cong \compexfull{\MedvedevDoctrine}$.
\end{proof}

The previous result allows us to establish a first link between the Dialectica doctrines \cite{trotta23TCS,trotta2022} and the doctrines for computability presented in the previous sections, employing categorical universal properties.

We recall that Dialectica categories were originally introduced in \cite{dePaiva1989dialectica} as a categorification of G\"odel's Dialectica interpretation \cite{goedel1986}. Over the years, several authors have noticed some resemblance between the structure of Dialectica categories and some known notions of computability. However, despite the outwardly similar appearance, a formal connection between computability and Dialectica categories has never been proved so far.

To establish such a link, we employ a result due to Hofstra \cite{Hofstra2011}. In particular, he proved that the dialectica construction can be presented by combining the pure universal and pure existential completion (in the categorical setting of fibrations). Hence, in the language of doctrines, this means that the doctrine $\dial{P}$ happens to be equivalent to the doctrine $\compex{(\compun{P})}$. 

Following Hofstra's intuition, we can consider a natural generalization of the Dialectica construction, namely the \emph{full Dialectica construction}: given a doctrine $P$ we define the full Dialectica construction $\dialfull{P}$ as the doctrine $\compexfull{(\compunfull{P})}$.

Therefore, as direct corollary of \thref{thm:D_is_full_ex_comp_Med} and \thref{thm:medvedev_iso}, we obtain the following:
\begin{corollary}
 We have an isomorphism of doctrines $\D\cong \dialfull{\Mpcadoctrine}$.
\end{corollary}

\begin{remark}
Observe that the preorder $\le_D$ describes a pointwise connection between $F$ and $G$. More precisely, once we fix $x\in X$, for every $A$ such that $x\name_R A$, the map $\varphi$ selects some $B$ such that $x\name_S B$. In other words, in order to solve $F(x,A)$ we need to look at $G(x,\varphi(x,A))$. However, the definition of extended-strong-Weihrauch reducibility allows for more flexibility: once we fix $x\in X$, we can use the map $\ovln{k}$ to prompt $G$ on a (possibly) different input $(\ovln{k}(x),\varphi(x,A))$. The gap between $\extstrongweireducible$ and $\le_D$ can therefore be bridged by requiring the existence of an effective map $\ovln{k}$ transforming the input for $F$ to a (possibly different) input for $G$.
\end{remark}
\begin{theorem}
    For every $\ovln{k} \in \subpca{A}$ and every $G\in \D(\pca{A})$, let $G_{\ovln{k}}$ denote the map $(p,A)\mapsto G(\ovln{k}\cdot p,A)$. For every $F,G\in \D(\pca{A})$, 
    \[ F\extstrongweireducible G \iff (\exists \ovln{k} \in \subpca{A})( F \le_D G_{\ovln{k}}). \]
\end{theorem}
\begin{proof}
    This is immediate by unfolding the definitions: $F\extstrongweireducible G$ via $(\ovln{k},\varphi),\ovln{h}$ iff $\varphi$ and $\ovln{h}$ witness the reduction $F\le_D G_{\ovln{k}}$. 
\end{proof}

\section{Conclusions}
\label{sec:conclusions}
We set out to and managed to categorify the notions of Turing, Medvedev, Muchnik, and Weihrauch reducibility and their variants.

To show the categorification works, we proved the Medvedev isomorphism theorem, the Muchnik isomorphism theorem, and the Weihrauch isomorphism theorems.
We also showed how the respective Medvedev, Muchnik, and Weihrauch doctrines relate to existential and universal completions.  Then, we consider several generalizations of Weihrauch reducibility for represented spaces and assemblies, extending our previous approach to these settings.
In future work, we aim to delve more deeply into the algebraic structures of these doctrines. If the structure proves to be sufficiently suitable, we intend to apply the tripos-to-topos construction to these doctrines and study the resulting categories.

\bibliographystyle{alphaurl}
\bibliography{references}

\end{document}